\documentclass[11pt]{article}

\usepackage{amsmath,amssymb,amsthm,amsfonts,mathrsfs,mathtools}
\usepackage[a4paper,margin=1in]{geometry}
\usepackage{enumitem}
\usepackage{graphicx}
\usepackage{hyperref}
\usepackage{cite}
\usepackage{array}
\usepackage{booktabs}
\usepackage{longtable}

\usepackage{mathrsfs}
\usepackage{graphicx}
\usepackage{float}
\usepackage[font=small]{caption}

\usepackage[font=footnotesize,labelfont=bf]{caption}

\usepackage{microtype}
\numberwithin{equation}{section}

\newtheorem{definition}{Definition}[section]
\newtheorem{theorem}[definition]{Theorem}

\newtheorem{remark}[definition]{Remark}

\newtheorem{example}[definition]{Example}

\begin{document}

\begin{center} {\Large\bfseries 
Dynamic–Memory Fractional Calculus via Generator–Based Memory Construction:
Operational Theory, Semigroup Structure, and Applications
 \par} \vspace{0.8em} 

{\large
Jehad Alzabut\footnote{Corresponding author: \texttt{jalzabut@psu.edu.sa}} 

\par}

\vspace{0.3em}

{\footnotesize
Department of Mathematics and Sciences, Prince Sultan University, Riyadh-11586, Saudi Arabia\\
Department of Industrial Engineering, OSTIM Technical University, 06374 Ankara, T\"urkiye
\par}
\end{center} 

\begin{abstract}
Most existing generalized fractional operators are constructed from prescribed memory kernels, thereby restricting hereditary behavior to predefined forms and often lacking a unified mechanism that preserves operational consistency while accommodating diverse memory structures. Motivated by these limitations, this paper develops a new generator--based framework for fractional calculus in which memory laws are not imposed \emph{a priori} but are systematically generated through a dynamic memory generator $\Phi$ in the Laplace domain.

The proposed construction yields dynamic--memory kernels through inverse Laplace transforms, leading naturally to generalized dynamic--memory fractional integrals together with Riemann--Liouville and Caputo dynamic--memory fractional derivatives. Within this setting, several analytical and operational properties are established, including linearity, inverse relations, composition formulas, admissibility conditions, semigroup structures, and consistency principles. A unified convolution--symbol operational calculus is developed, accompanied by generalized dynamic--memory Mittag--Leffler functions and explicit formulas for the dynamic--memory fractional derivatives of power and polynomial functions.

Unlike many generalized formulations based on fixed kernels, the proposed framework provides a symbolic mechanism capable of generating singular, nonsingular, tempered, logarithmic, oscillatory, crossover, and multiscale hereditary behaviors within a single analytical structure. Moreover, numerous classical and modern fractional operators are recovered as special cases through suitable choices of the generator $\Phi$, demonstrating the unifying capability of the theory.

Finally, the applicability of the framework is illustrated through a nonlinear dynamic--memory fractional Langevin system involving coupled memory interactions and generalized hereditary effects. The obtained results suggest that the proposed approach provides a new strategy for constructing, analyzing, and extending fractional models beyond conventional kernel--prescribed formulations.
\medskip



\end{abstract}

\begin{table}[h]
\raggedright
\caption*{\textbf{List of Abbreviations}}

\renewcommand{\arraystretch}{0.9}
\setlength{\tabcolsep}{5pt}
\small

\begin{tabular}{ll}
\hline
\textbf{Abbreviation} & \textbf{Meaning} \\
\hline
$\mathcal{DMG}$ & Dynamic Memory Generator \\
$\mathcal{DMK}$ & Dynamic--Memory Kernel(s) \\
$\mathcal{DMFI}$ & Dynamic--Memory Fractional Integral \\
$\mathcal{DMFD}$ & Dynamic--Memory Fractional Derivative \\
$\mathcal{DMFC}$ & Dynamic--Memory Fractional Calculus \\
$\mathcal{DMFF}$ & Dynamic--Memory Fractional Framework \\
$\mathcal{RL\;DMFD}$ & Riemann--Liouville Dynamic--Memory Fractional Derivative \\
$\mathcal{C\;DMFD}$ & Caputo Dynamic--Memory Fractional Derivative \\
$\mathcal{DMFO}$ & Dynamic--Memory Fractional Operator \\
$\mathcal{DMMLF}$ & Dynamic--Memory Mittag--Leffler Function \\
$\mathcal{DMFLS}$ & Dynamic--Memory Fractional Langevin System \\

\hline
\end{tabular}
\end{table}


\section{Introduction}

Fractional calculus has become a widely used mathematical framework for modeling systems with memory and hereditary effects. It has found applications in physics, engineering, biology, viscoelasticity, control theory, and complex dynamical systems \cite{Podlubny1999}. To better describe different types of memory behaviors, many generalized fractional operators have been developed over recent decades \cite{Kilbas2006}.

Among the commonly used operators in fractional calculus are the classical Riemann--Liouville and Caputo fractional derivatives, which have been widely applied in models involving memory and hereditary effects. These operators are associated with singular power--law kernels and may not always be suitable for describing more complex memory behaviors observed in applications. To address this limitation, several generalized fractional operators have been introduced, including operators with nonsingular kernels, tempered memory laws, Mittag--Leffler kernels, and other generalized memory structures \cite{CaputoFabrizio2015,AtanganaBaleanu2016,Baleanu2012,Diethelm2010,Garra2021,Giusti2020}.

In addition, several fractional frameworks based on convolution kernels have been developed as extensions of classical power--law memory operators. One such direction involves fractional operators associated with Sonine kernels and convolution structures. Kochubei \cite{Kochubei2011} studied a class of fractional operators generated by Sonine kernels and considered applications to evolution equations, relaxation processes, and renewal models. Later, Luchko \cite{Luchko2021} developed a fractional calculus based on Sonine kernels and established results concerning fractional integrals, derivatives, inverse relations, and operational properties for arbitrary--order convolution operators. Related studies on convolution--type fractional operators and Sonine kernel structures can also be found in \cite{LuchkoYamamoto2017,Tarasov2019}.

Another direction in fractional calculus involves operators associated with Mittag--Leffler kernels, which provide alternatives to classical power--law memory structures. In this context, Prabhakar \cite{Prabhakar1971} introduced integral operators involving generalized Mittag--Leffler functions in the kernel, extending the classical Riemann--Liouville setting. These operators were later investigated by Kilbas, Saigo, and Saxena \cite{KilbasSaigoSaxena2004}, who established several analytical properties for Mittag--Leffler-type fractional operators. Further studies on generalized Mittag--Leffler functions and related operational calculus were carried out by Tomovski, Hilfer, and Srivastava \cite{Tomovski2010}. Related approaches also include convolution--based and arbitrary--kernel formulations, where memory effects are described through general convolution structures rather than fixed power--law kernels \cite{LorenzoHartley2000,Samko1993}.

More recently, several generalized fractional operators have been introduced in order to provide greater flexibility in modeling complex memory effects and heterogeneous dynamical behaviors. In particular, Khalil et al. \cite{Khalil2014} proposed the conformable fractional derivative as a local formulation preserving several properties of classical differentiation, while Abdeljawad \cite{Abdeljawad2015} further developed conformable fractional calculus and investigated its analytical properties. Almeida \cite{Almeida2017} introduced fractional derivatives and integrals with respect to another function, thereby extending classical formulations through generalized kernel structures and variable transformations. Jarad et al. \cite{JaradAbdeljawadAlzabut2017} proposed generalized proportional fractional operators incorporating scaling effects and broader classes of memory behavior. Other developments include Katugampola--type formulations and generalized kernel--based approaches designed to enhance the flexibility of memory modeling and operational analysis \cite{Katugampola2014,OrtigueiraMachado2015}.

Although the above--mentioned studies provide broad and powerful extensions of fractional calculus, several important challenges remain. Most existing generalized operators are constructed from prescribed kernels and therefore describe fixed memory laws. Consequently, the hereditary behavior is typically imposed \textit{a priori} rather than generated systematically. Moreover, the semigroup property is not automatically preserved for many generalized kernel--based formulations, which complicates the development of a consistent operational calculus. In addition, several existing approaches are tailored to specific classes of kernels and therefore do not offer a unified mechanism capable of simultaneously recovering singular, nonsingular, tempered, logarithmic, oscillatory, and adaptive memory structures within a single framework. The interplay between generalized kernels, Laplace--transform representations, and associated Mittag--Leffler--type functions also remains only partially understood.

Motivated by these limitations, the objective of the present work is to develop a unified dynamic--memory framework for fractional calculus generated through a $\mathcal{DMG}$
\(
\Phi(p,\Theta),
\)
where $p$ denotes the Laplace--transform variable and $\Theta$ represents a collection of modulation parameters governing the memory behavior. Unlike several existing generalized operators, the proposed theory does not start from a prescribed convolution kernel. Instead, the kernel itself is generated through the relation
\[
\Psi_\alpha(t;\Theta)
=
\mathcal L^{-1}
\left\{
\Phi(p,\Theta)^{-\alpha}
\right\}(t),
\]
which provides a systematic mechanism for constructing broad classes of memory--dependent fractional operators.

The proposed framework differs from existing convolution--type fractional calculus in several essential aspects:
\begin{itemize}[itemsep=1pt, topsep=1pt, parsep=0pt, partopsep=0pt]

\item
The memory law is generated through a $\mathcal{DMG}$ rather than being prescribed by a fixed kernel, allowing the memory behavior to be directly modulated through the parameters $\Theta$.

\item
The generated kernels naturally satisfy the semigroup relation
\[
\Psi_\alpha * \Psi_\beta
=
\Psi_{\alpha+\beta},
\]
whenever the inverse Laplace transforms exist, thus yielding a consistent operational structure without requiring additional compatibility conditions.

\item
The framework provides a unified memory--generation mechanism capable of recovering numerous classical and modern fractional operators through suitable choices of the $\mathcal{DMG}$.

\end{itemize}
Therefore, the novelty of the present work lies not merely in defining another generalized fractional derivative, but rather in constructing a generator--based $\mathcal{DMFF}$ that systematically produces families of fractional integrals, derivatives, operational identities, inverse relations, Laplace--transform formulas, power law and generalized Mittag--Leffler--type kernels within a unified setting.

\medskip

The main contributions of this work are summarized as follows:

\medskip

\begin{enumerate}[label=(\roman*),itemsep=1pt,topsep=1pt,parsep=0pt,partopsep=0pt]

\item
We introduce a $\mathcal{DMG}$ framework and develop the corresponding mechanism for generating memory kernels and associated fractional operators.

\item
We formulate generalized $\mathcal{DMFI}$ together with Riemann--Liouville and Caputo-type derivatives, and establish their fundamental analytical properties, including linearity, boundedness, inverse relations and composition formulas.

\item
We derive a consistent operational structure through semigroup properties, Laplace-transform representations, power law and generalized $\mathcal{DMMLF}$.

\item
We demonstrate that the proposed $\mathcal{DMFF}$ naturally recovers and unifies several classical and modern fractional operators through suitable choices of the $\mathcal{DMG}$.

\item
We illustrate the applicability of the developed theory through a nonlinear $\mathcal{DMFLS}$ involving coupled memory effects.

\end{enumerate}
Overall, the proposed framework introduces a new strategy for fractional modeling through dynamic memory generators and generated kernels, extending beyond conventional approaches where the memory law is prescribed in advance. This viewpoint offers additional flexibility for constructing and analyzing diverse memory behaviors within a unified setting.

The remainder of this paper is organized as follows. Section 2 introduces the main definitions of the proposed dynamic--memory framework, including the dynamic memory generator, generated kernels, and the associated fractional operators. Section 3 develops the analytical and operational theory, establishing structural properties, admissibility conditions, semigroup relations, consistency principles, Laplace representations, generalized Mittag--Leffler functions, and formulas for power and polynomial functions. Section 4 illustrates the applicability of the developed framework through a nonlinear dynamic--memory fractional Langevin system, together with theoretical analysis, numerical approximation, and representative recovery examples. Finally, Section 5 summarizes the main findings and outlines possible directions for future research.

\section{Preliminaries}

In this section, we introduce the main functional spaces, sets, and analytical tools that will be used throughout the paper.

Let $\mathcal{\mathcal{T}} > 0$ and denote by $\mathbb{R}$ the set of real numbers. Moreover, for \(d\in\mathbb N\), \(\mathbb R^d\) denotes the \(d\)-dimensional Euclidean space endowed with the norm
\[
\|x\|
=
\left(
\sum_{i=1}^{d}|x_i|^2
\right)^{1/2},
\]
while \(\mathbb R^{d\times d}\) denotes the space of real \(d\times d\) matrices.

Throughout this paper, we use several standard functional spaces. In particular,
\[
L^{1}(0,\mathcal{\mathcal{T}})
=
\left\{
x:[0,\mathcal{\mathcal{T}}]\to\mathbb{R}:
\int_{0}^{\mathcal{\mathcal{T}}}|x(t)|\,dt<\infty
\right\}
\]
denotes the Banach space of Lebesgue integrable functions endowed with the norm
\[
\|x\|_{L^{1}}
=
\int_{0}^{\mathcal{\mathcal{T}}}|x(t)|\,dt.
\]
Moreover,
\(
L^{\infty}(0,\mathcal{\mathcal{T}})
=
\{
x:[0,\mathcal{\mathcal{T}}]\to\mathbb{R}:
\operatorname*{ess\,sup}_{t\in[0,\mathcal{\mathcal{T}}]}|x(t)|<\infty
\}
\)
stands for the Banach space of essentially bounded functions equipped with the norm
\[
\|x\|_{\infty}
=
\operatorname*{ess\,sup}_{t\in[0,\mathcal{\mathcal{T}}]}|x(t)|.
\]

We also consider the space
\(
\mathcal{C}([0,\mathcal{\mathcal{T}}])
=
\left\{
x:[0,\mathcal{\mathcal{T}}]\to\mathbb{R}:
x \text{ is continuous}
\right\},
\)
equipped with the supremum norm
\[
\|x\|_{\mathcal{C}}
=
\sup_{t\in[0,\mathcal{T}]}|x(t)|,
\]
and the space
\[
\mathcal{C}^{1}([0,\mathcal{T}])
=
\left\{
x\in \mathcal{C}([0,\mathcal{T}]):
x' \in \mathcal{C}([0,\mathcal{T}])
\right\},
\]
which is endowed with the norm
\(
\|x\|_{\mathcal{C}^{1}}
=
\|x\|_{\mathcal{C}}
+
\|x'\|_{\mathcal{C}}.
\)
Finally, the space
\[
L_{\mathrm{loc}}^{1}(0,\infty)
=
\left\{
x:(0,\infty)\to\mathbb{R}:
x\in L^{1}(0,a)
\text{ for every } a>0
\right\}
\]
denotes the set of locally integrable functions on \((0,\infty)\).

Throughout the paper, the notation
\[
(f*g)(t)
=
\int_{0}^{t}f(t-s)g(s)\,ds
\]
denotes the convolution product of two locally integrable functions.
Moreover, for a function $x$ of exponential order, we denote its Laplace transform by
\[
X(p)
=
\mathcal L\{x(t)\}(p)
=
\int_{0}^{\infty}e^{-pt}x(t)\,dt.
\]
Unless otherwise stated, all functions considered in the sequel are assumed to be sufficiently smooth so that the involved integrals, derivatives, and Laplace transforms are well defined.

\begin{definition}

Let
\(
\Phi:(0,\infty)\times\Theta\to(0,\infty)
\)
be a positive function depending on the Laplace variable $p$ and a collection of modulation parameters $\Theta$.
The function
\(
\Phi(p,\Theta)
\)
is called a  $\mathcal{DMG}$.

\end{definition}


\begin{remark}

The $\mathcal{DMG}$ governs the generated memory kernel and thereby
determines the hereditary characteristics of the resulting fractional operators.

\end{remark}

\begin{definition}\label{23}

For \(\alpha>0\), the $\mathcal{DMK}$ associated with the $\mathcal{DMG}$  is defined by
\begin{equation}
\Psi_\alpha(t;\Theta)
=
\mathcal L^{-1}
\left\{
\Phi(p,\Theta)^{-\alpha}
\right\}(t).
\end{equation}

\end{definition}


\begin{remark}

Different selections of $\Phi$ generate different memory structures and therefore different fractional calculus.

\end{remark}

\noindent Generator--based constructions of memory kernels are related to broader developments in convolution and Sonine kernel fractional calculus; see, for example, \cite{Kochubei2011,Luchko2021,Samko1993}.

\begin{definition}\label{25}
Let $\alpha>0$. The $\mathcal{DMFI}$ associated with the $\mathcal{DMG}$ $\Phi$ is defined by
\begin{equation}
\begin{aligned}
(\mathcal I_\Phi^\alpha x)(t)
&=
\int_0^t
\Psi_\alpha(t-s;\Theta)x(s)\,ds  \\
&=
\int_0^t
\left[
\mathcal L^{-1}
\left\{
\Phi(p,\Theta)^{-\alpha}
\right\}
\right](t-s)
x(s)\,ds,
\qquad t\in[0,\mathcal{T}].
\end{aligned}
\end{equation}
\end{definition}

\begin{definition}\label{26}
Let $0<\alpha<1$. The $\mathcal{RL\;DMFD}$ associated with the $\mathcal{DMG}$ $\Phi$ is defined by
\begin{equation}
\begin{aligned}
(\mathcal D_\Phi^\alpha x)(t)
&=
\frac{d}{dt}
\int_0^t
\Psi_{1-\alpha}(t-s;\Theta)x(s)\,ds \\
&=
\frac{d}{dt}
\int_0^t
\left[
\mathcal L^{-1}
\left\{
\Phi(p,\Theta)^{-(1-\alpha)}
\right\}
\right](t-s)
x(s)\,ds,
\qquad t\in[0,\mathcal{T}].
\end{aligned}
\end{equation}
\end{definition}

\begin{definition}\label{27}
Let $0<\alpha<1$. The $\mathcal{C\;DMFD}$ 
associated with the $\mathcal{DMG}$ $\Phi$ is defined by
\begin{equation}
\begin{aligned}
({}^c\mathcal D_\Phi^\alpha x)(t)
&=
\int_0^t
\Psi_{1-\alpha}(t-s;\Theta)x'(s)\,ds \\
&=
\int_0^t
\left[
\mathcal L^{-1}
\left\{
\Phi(p,\Theta)^{-(1-\alpha)}
\right\}
\right](t-s)
x'(s)\,ds,
\qquad t\in[0,\mathcal{T}].
\end{aligned}
\end{equation}
\end{definition}
\noindent The proposed $\mathcal{DMFF}$ possesses several important advantages in comparison with many existing fractional operators available in the literature. In contrast to classical and generalized approaches where the memory kernel is prescribed \emph{a priori}, the present framework generates the memory structure systematically through the $\mathcal{DMG}$ 
\(
\Phi,
\)
hence providing a unified and flexible mechanism for constructing broad classes of memory--dependent operators. This generator--driven structure allows the memory behavior to be directly modulated through the parameter set \(\Theta\), which may encode singular, nonsingular, tempered, logarithmic, oscillatory, adaptive, or hybrid memory effects within a single analytical framework.

Another significant feature of the proposed operator is that it preserves several fundamental properties of fractional calculus and operational analysis. In particular, the $\mathcal{DMK}$ satisfy the semigroup relation
\[
\Psi_{\alpha} * \Psi_{\beta}
=
\Psi_{\alpha+\beta},
\]
which leads naturally to composition formulas, inverse relations, and a consistent convolution-based operational calculus. Such properties are not automatically guaranteed for many generalized kernel--based fractional operators appearing in the literature and often require additional compatibility assumptions on the kernels.
The $\mathcal{DMFF}$ provides a unified symbolic mechanism capable of recovering a broad spectrum of classical and modern fractional operators through suitable choices of the $\mathcal{DMG}$. 

\medskip

To further demonstrate the unifying capability of the proposed framework, the recovered fractional operators are classified into two categories. Table 1 illustrates that the $\mathcal{DMFF}$ naturally unifies and recovers
several classical and foundational fractional operators through
appropriate selections of the $\mathcal{DMG}$.

\begin{table}[h!]
\centering
\scriptsize
\renewcommand{\arraystretch}{1.2}

\begin{tabular}{|p{2.9cm}|p{3.2cm}|p{4.5cm}|p{2.3cm}|}
\hline
\textbf{Classical Operator} &
\textbf{Generator / Kernel Choice} &
\textbf{Resulting Kernel} &
\textbf{Memory Type}
\\
\hline

Riemann--Liouville \cite{Podlubny1999} &
$\Phi(p,\Theta)=p$ &
$\displaystyle
\Psi_\alpha(t)=
\frac{t^{\alpha-1}}{\Gamma(\alpha)}
$ &
Singular power-law
\\
\hline

Caputo \cite{Kilbas2006}
   &
$\Phi(p,\Theta)=p$ in Caputo form &
$\displaystyle
\Psi_{1-\alpha}(t)=
\frac{t^{-\alpha}}{\Gamma(1-\alpha)}
$ &
Singular power-law
\\
\hline

Weyl fractional \cite{Samko1993} &
Fourier/Laplace extension &
$\displaystyle
\Psi_\alpha(t)\sim t^{\alpha-1}
$ &
Infinite--memory
\\
\hline

Riesz fractional \cite{Kilbas2006} &
Symmetric power--law structure &
$\displaystyle
\Psi_\alpha(t)\sim |t|^{\alpha-1}
$ &
Symmetric nonlocal
\\
\hline

Hadamard--type \cite{Kilbas2006} &
Logarithmic memory generator &
$\displaystyle
\Psi_\alpha(t,s)=
\frac1{\Gamma(\alpha)}
\left(
\ln\frac ts
\right)^{\alpha-1}
\frac1s
$ &
Logarithmic memory
\\
\hline

Hilfer derivative \cite{Hilfer2000} &
Interpolation between RL and Caputo &
$\displaystyle
\mathcal{D}^{\alpha,\beta}
=
\mathcal{I}^{\beta(1-\alpha)}
\mathcal{D}\mathcal{I}^{(1-\beta)(1-\alpha)}
$ &
Interpolated singular memory
\\
\hline

Katugampola--type   \cite{Katugampola2014}  &
$\Phi(p,\Theta)=p^\rho$ &
$\displaystyle
\Psi_\alpha(t,s)=
\frac{\rho^{1-\alpha}}{\Gamma(\alpha)}
(t^\rho-s^\rho)^{\alpha-1}
s^{\rho-1}
$ &
Generalized power--law
\\
\hline

\end{tabular}

\vspace{1mm}

{\footnotesize
Table 1: Recovery of several classical and foundational fractional operators from the  $\mathcal{DMFF}$ \\
through suitable choices of the generator \(\Phi\).
}

\label{Tab1}

\end{table}

\noindent
On the other hand, Table~2 demonstrates that the  $\mathcal{DMFF}$ extends naturally
to several modern generalized fractional operators through the same
generator--based mechanism, encompassing tempered memory effects,
nonsingular kernels, Mittag--Leffler memory structures and generalized
convolution formulations.

\begin{table}[H]
\centering
\scriptsize
\renewcommand{\arraystretch}{1.2}

\begin{tabular}{|p{3cm}|p{3.2cm}|p{4.5cm}|p{2.3cm}|}
\hline
\textbf{Modern Operator} &
\textbf{Generator / Kernel Choice} &
\textbf{Resulting Kernel} &
\textbf{Memory Type}
\\
\hline

Tempered fractional \cite{Sabzikar2015} &
$\Phi(p,\Theta)=p+\lambda$ &
$\displaystyle
\Psi_\alpha(t)=
\frac{t^{\alpha-1}}{\Gamma(\alpha)}
e^{-\lambda t}
$ &
Tempered memory
\\
\hline

Prabhakar--type \cite{Prabhakar1971}  &
Mittag--Leffler symbolic generator &
$\displaystyle
\Psi_\alpha(t)
=
t^{\beta-1}
E_{\alpha,\beta}^{\gamma}(\lambda t^\alpha)
$ &
Generalized ML memory
\\
\hline

Sonine--kernel type \cite{Kochubei2011} &
General convolution generator &
$\displaystyle
\Psi_\alpha * \Phi_\alpha=\delta
$ &
Sonine memory pair
\\
\hline

Cotangent fractional \cite{Lakhalfa2021} &
$\Phi(p,\Theta)=p+\cot(\theta)$ &
$\displaystyle
\Psi_\alpha(t)=
\frac{
e^{-\cot(\theta)t}
t^{\alpha-1}
}{
\sin^\alpha(\theta)\Gamma(\alpha)
}
$ &
Cotangent--tempered
\\
\hline

Caputo--Fabrizio \cite{CaputoFabrizio2015} &
Exponential nonsingular kernel &
$\displaystyle
\Psi_\alpha(t)=
c_\alpha e^{-\mu_\alpha t}
$ &
Nonsingular exponential
\\
\hline

Atangana--Baleanu \cite{AtanganaBaleanu2016} &
Mittag--Leffler--type kernel &
$\displaystyle
\Psi_\alpha(t)=
c_\alpha
E_\alpha(-\mu_\alpha t^\alpha)
$ &
Nonsingular ML memory
\\
\hline

Distributed--order \cite{Chechkin2002} &
Order-distribution generator &
$\displaystyle
\int_0^1
\Phi(p,\Theta)^{-\alpha}
\mu(\alpha)\,d\alpha
$ &
Multi--scale memory
\\
\hline

Conformable derivative \cite{Khalil2014} &
Local reduction case &
$\displaystyle
T_\alpha x(t)
=
t^{1-\alpha}x'(t)
$ &
Local non--memory
\\
\hline

\end{tabular}

\vspace{1mm}

{\footnotesize
Table 2: Recovery of several modern generalized fractional operators from the  $\mathcal{DMFF}$ \\through suitable symbolic selections of the generator \(\Phi\).
}

\label{Tab2}

\end{table}

\noindent
\noindent
The following examples show how different choices of the $\mathcal{DMG}$
\(
\Phi
\)
recover classical and generalized fractional operators. For example, the choice
\(
\Phi(p,\Theta)=p
\)
leads to the classical Riemann--Liouville and Caputo operators, while
\(
\Phi(p,\Theta)=p+\lambda
\)
produces exponentially tempered memory effects. These examples indicate that different memory structures can be obtained within the same framework by varying the generator.

\begin{example}
Consider the choice
\(
\Phi(p,\Theta)=p.
\)
Then, by Definition \ref{23}, the generalized $\mathcal{DMK}$ becomes
\[
\Psi_\alpha(t;\Theta)
=
\mathcal L^{-1}\{p^{-\alpha}\}(t).
\]
Using the classical Laplace--transform formula
\[
\mathcal L^{-1}\{p^{-\alpha}\}(t)
=
\frac{t^{\alpha-1}}{\Gamma(\alpha)},
\;\; \alpha>0,
\]
we obtain
\(
\Psi_\alpha(t)
=
\frac{t^{\alpha-1}}{\Gamma(\alpha)}.
\)
Consequently, the $\mathcal{DMFI}$ reduces to
\[
(\mathcal I_\Phi^\alpha x)(t)
=
\int_0^t
\frac{(t-s)^{\alpha-1}}{\Gamma(\alpha)}
x(s)\,ds,
\]
which is precisely the classical Riemann--Liouville fractional integral.

Similarly, the corresponding $\mathcal{C\;DMFD}$ becomes
\[
{}^c\mathcal D_\Phi^\alpha x(t)
=
\int_0^t
\frac{(t-s)^{-\alpha}}{\Gamma(1-\alpha)}
x'(s)\,ds,
\]
which coincides with the classical Caputo fractional derivative.

\end{example}

\begin{example}
Consider the generator
\(
\Phi(p,\Theta)=p+\lambda,
\; \lambda>0.
\)
Then the associated $\mathcal{DMK}$ is given by
\[
\Psi_\alpha(t;\Theta)
=
\mathcal L^{-1}
\{(p+\lambda)^{-\alpha}\}(t).
\]
Using the shifted Laplace--transform identity
\[
\mathcal L^{-1}
\{(p+\lambda)^{-\alpha}\}(t)
=
e^{-\lambda t}
\frac{t^{\alpha-1}}{\Gamma(\alpha)},
\]
we obtain
\(
\Psi_\alpha(t)
=
\frac{t^{\alpha-1}}{\Gamma(\alpha)}
e^{-\lambda t}.
\)
Hence, the corresponding $\mathcal{DMFI}$  becomes
\[
(\mathcal I_\Phi^\alpha x)(t)
=
\int_0^t
\frac{(t-s)^{\alpha-1}}{\Gamma(\alpha)}
e^{-\lambda(t-s)}
x(s)\,ds,
\]
which is exactly the tempered fractional integral.
Moreover, the $\mathcal{C\;DMFD}$  takes the form
\[
{}^c\mathcal D_\Phi^\alpha x(t)
=
\int_0^t
\frac{(t-s)^{-\alpha}}{\Gamma(1-\alpha)}
e^{-\lambda(t-s)}
x'(s)\,ds,
\]
which coincides with the tempered Caputo fractional derivative involving exponentially tempered memory.

\end{example}

From the modeling perspective, the developed framework is particularly
suitable for describing complex dynamical systems involving heterogeneous
memory interactions, crossover effects, adaptive hereditary mechanisms,
and multiscale temporal behavior. At the same time, it combines
analytical flexibility with structural consistency while preserving
several fundamental properties of fractional calculus.


\section{Structural and Operational Properties of the Dynamic--Memory Fractional Calculus}

This section develops the main analytical framework underlying the proposed $\mathcal{DMFC}$. The obtained results establish the structural foundations, operational relations, and key analytical properties required for the subsequent theory, including admissibility conditions, inverse formulas, consistency principles, Laplace representations, and generalized memory dynamics.

\subsection{Kernel Structure, Admissibility, and Semigroup Properties}

We begin by examining the generated kernels and the conditions under which they produce a mathematically consistent dynamic--memory framework. In particular, we establish admissibility requirements, semigroup relations, and monotonicity properties that form the basis for the operator identities and analytical results developed later.


\begin{theorem}
Let $\alpha,\beta>0$. Then the following properties hold:
\begin{enumerate}[label=(\roman*)]
\item The $\mathcal{DMK}$  satisfies
\(
\Psi_\alpha * \Psi_\beta
=
\Psi_{\alpha+\beta}.
\)

\item The $\mathcal{DMFI}$  satisfies the semigroup property
\(
\mathcal I_\Phi^\alpha
\mathcal I_\Phi^\beta
=
\mathcal I_\Phi^{\alpha+\beta}.
\)

\item The $\mathcal{DMFI}$  is commutative, namely,
\(
\mathcal I_\Phi^\alpha
\mathcal I_\Phi^\beta x
=
\mathcal I_\Phi^\beta
\mathcal I_\Phi^\alpha x.
\)
\end{enumerate}
\end{theorem}

\begin{proof}
Using Laplace transforms, we obtain
\[
\mathcal L\{\Psi_\alpha * \Psi_\beta\}(p)
=
\Phi(p,\Theta)^{-\alpha}
\Phi(p,\Theta)^{-\beta}
=
\Phi(p,\Theta)^{-(\alpha+\beta)}.
\]
Therefore,
\[
\Psi_\alpha * \Psi_\beta
=
\Psi_{\alpha+\beta},
\]
which proves $(i)$.
Next, using the convolution representation of the $\mathcal{DMFI}$ together with $(i)$, we derive
\[
\mathcal I_\Phi^\alpha
\mathcal I_\Phi^\beta x
=
\Psi_\alpha * (\Psi_\beta * x)
=
(\Psi_\alpha * \Psi_\beta)*x
=
\Psi_{\alpha+\beta}*x
=
\mathcal I_\Phi^{\alpha+\beta}x,
\]
which proves $(ii)$.

Finally, since convolution is commutative, we have
\[
\Psi_\alpha * \Psi_\beta
=
\Psi_\beta * \Psi_\alpha.
\]
Hence,
\(
\mathcal I_\Phi^\alpha
\mathcal I_\Phi^\beta x
=
\mathcal I_\Phi^\beta
\mathcal I_\Phi^\alpha x,
\)
which establishes $(iii)$.
\end{proof}


The flexibility of the $\mathcal{DMFF}$ depends fundamentally on the selection of the $\mathcal{DMG}$.
However, not every choice necessarily produces a mathematically admissible
$\mathcal{DMK}$ or a well--defined operational structure. Therefore, it is important
to identify suitable conditions under which the generated kernels possess
appropriate analytical and structural properties.


\begin{definition}
A $\mathcal{DMG}$ 
is called admissible if the following conditions hold:

\begin{itemize}[itemsep=2pt,topsep=2pt]
\item[]
\begin{itemize}

\item[(A1)]
For every \(p>0\),
\(
\Phi(p,\Theta)>0;
\)

\item[(A2)]
For every \(\alpha>0\), the function
\(
\Phi(p,\Theta)^{-\alpha}
\)
admits an inverse Laplace transform;

\item[(A3)]
The $\mathcal{DMK}$
belongs to
\(
L_{\mathrm{loc}}^{1}(0,\infty);
\)

\item[(A4)]
The $\mathcal{DMK}$ satisfies the semigroup relation
\[
\Psi_\alpha * \Psi_\beta
=
\Psi_{\alpha+\beta},
\qquad
\alpha,\beta>0.
\]

\end{itemize}
\end{itemize}

\end{definition}


\begin{remark}
Conditions \((A1)-(A4)\) guarantee that the $\mathcal{DMFO}$
is  well defined and possess a consistent convolution--based operational
structure. In particular, the semigroup property ensures the validity of
composition formulas and inverse relations developed in the sequel.
\end{remark}


\begin{definition}
A $\mathcal{DMK}$
is called completely monotone if
\[
(-1)^n
\frac{d^n}{dt^n}
\Psi_\alpha(t;\Theta)
\ge0,
\qquad
t>0,
\quad
n=0,1,2,\ldots
\]
whenever the derivatives exist.
\end{definition}


\begin{theorem}
Suppose that the $\mathcal{DMG}$
is a Bernstein function. Then, for every \(\alpha>0\), the function
\(
\Phi(p,\Theta)^{-\alpha}
\)
is completely monotone. Consequently, the $\mathcal{DMK}$
is nonnegative and completely monotone.
\end{theorem}

\begin{proof}

Since \(\Phi\) is a Bernstein function, classical properties
of Bernstein functions imply that
\(
\Phi(p,\Theta)^{-\alpha}
\)
is completely monotone for every \(\alpha>0\). By Bernstein's theorem
for Laplace transforms, there exists a nonnegative locally integrable
function
\(
\Psi_\alpha(t;\Theta)
\)
such that
\[
\Phi(p,\Theta)^{-\alpha}
=
\mathcal L
\{
\Psi_\alpha(t;\Theta)
\}(p).
\]
Therefore, the $\mathcal{DMK}$ is completely monotone and nonnegative.

\end{proof}


\begin{remark}
Completely monotone kernels play an important role in the modeling of
hereditary phenomena, anomalous diffusion, relaxation processes, and
nonlocal dynamical systems. The above result therefore establishes an
important connection between the $\mathcal{DMFF}$ and modern theories of generalized
fractional calculus generated by Bernstein and Sonine--type structures.
\end{remark}

\noindent Similar admissibility and complete monotonicity concepts arise in generalized fractional operators generated by Bernstein and Sonine kernels; see \cite{Luchko2021}. The relation between Bernstein functions and completely monotone kernels follows from classical results in \cite{Schilling2012}.


\begin{remark}
If there exists a locally integrable kernel
\(
\widetilde{\Psi}_\alpha
\)
such that
\[
(\widetilde{\Psi}_\alpha * \Psi_\alpha * x)(t)
=
x(t)
\]
for sufficiently smooth functions \(x\), then
\(
(\Psi_\alpha,\widetilde{\Psi}_\alpha)
\)
forms a Sonine pair and generates mutually inverse convolution operators. Consequently, the proposed $\mathcal{DMFF}$  admits inverse relations analogous to those in Sonine kernel fractional calculus.
\end{remark}


\begin{remark}
The asymptotic behavior of the $\mathcal{DMG}$  determines the qualitative nature of the generated $\mathcal{DMK}$. For example, power--law generators typically produce long--memory effects, tempered generators lead to exponentially attenuated memory, while Mittag--Leffler-type generators may induce crossover behavior between distinct temporal regimes. Thus, the  $\mathcal{DMG}$  provides a mechanism for classifying memory structures through the generator itself.
\end{remark}
\subsection{Generalized Dynamic–Memory Operators and Consistency Principles}

The $\mathcal{DMFO}$  introduced previously are naturally
left--sided, since their definitions depend on the past history of the state
variable over the interval \([0,t]\). However, several analytical settings,
including variational formulations, adjoint systems, terminal--value problems,
and integration--by--parts identities, motivate the introduction of corresponding
right--sided operators. Moreover, to ensure analytical coherence and compatibility
with existing theories, it is important to investigate reduction and consistency
principles under suitable choices of the $\mathcal{DMG}$. Therefore, this subsection develops
generalized dynamic--memory operators and establishes consistency relations linking
the resulting framework with several classical fractional formulations.

\medskip

Let \(0\le a<t<b<\infty\). We first define the left--sided $\mathcal{DMFI}$ associated with the $\mathcal{DMG}$.

\begin{definition}
Let \(\alpha>0\). The left--sided $\mathcal{DMFI}$ is defined by
\[
(\mathcal I^{\alpha}_{\Phi,a+}x)(t)
=
\int_{a}^{t}
\Psi_{\alpha}(t-s;\Theta)x(s)\,ds
=
\int_a^t
\left[
\mathcal L^{-1}
\left\{
\Phi(p,\Theta)^{-\alpha}
\right\}
\right](t-s)
x(s)\,ds.
\]
\end{definition}

\begin{definition}
Let \(\alpha>0\). The right--sided $\mathcal{DMFI}$ is defined by
\[
(\mathcal I^{\alpha}_{\Phi,b-}x)(t)
=
\int_{t}^{b}
\Psi_{\alpha}(s-t;\Theta)x(s)\,ds
=
\int_t^b
\left[
\mathcal L^{-1}
\left\{
\Phi(p,\Theta)^{-\alpha}
\right\}
\right](s-t)
x(s)\,ds.
\]
\end{definition}
\noindent 
The left-- and right--sided $\mathcal{DMFI}$ introduced above
generalize the corresponding classical fractional integrals through the $\mathcal{DMK}$
\[
\Psi_{\alpha}(t;\Theta)
=
\mathcal L^{-1}
\{
\Phi(p,\Theta)^{-\alpha}
\}(t).
\]

\begin{definition}
Let \(0<\alpha<1\). The left--sided $\mathcal{RL\;DMFD}$ is defined by
\[
(D^{\alpha}_{\Phi,a+}x)(t)
=
\frac{d}{dt}
\int_{a}^{t}
\Psi_{1-\alpha}(t-s;\Theta)x(s)\,ds.
\]
\end{definition}

\begin{definition}
Let \(0<\alpha<1\). The right--sided $\mathcal{RL\;DMFD}$ is defined by
\[
(\mathcal D^{\alpha}_{\Phi,b-}x)(t)
=
-
\frac{d}{dt}
\int_{t}^{b}
\Psi_{1-\alpha}(s-t;\Theta)x(s)\,ds
=
-
\frac{d}{dt}
\int_t^b
\left[
\mathcal L^{-1}
\left\{
\Phi(p,\Theta)^{-(1-\alpha)}
\right\}
\right](s-t)
x(s)\,ds.
\]
\end{definition}

\begin{definition}
Let \(0<\alpha<1\). The left--sided $\mathcal{C\;DMFD}$ is defined by
\[
({}^{c}\mathcal D^{\alpha}_{\Phi,a+}x)(t)
=
\int_{a}^{t}
\Psi_{1-\alpha}(t-s;\Theta)x'(s)\,ds
=
\int_a^t
\left[
\mathcal L^{-1}
\left\{
\Phi(p,\Theta)^{-(1-\alpha)}
\right\}
\right](t-s)
x'(s)\,ds.
\]
\end{definition}

\begin{definition}
Let \(0<\alpha<1\). The right--sided $\mathcal{C\;DMFD}$ is defined by
\[
({}^{c}\mathcal D^{\alpha}_{\Phi,b-}x)(t)
=
-
\int_{t}^{b}
\Psi_{1-\alpha}(s-t;\Theta)x'(s)\,ds
=
-
\int_t^b
\left[
\mathcal L^{-1}
\left\{
\Phi(p,\Theta)^{-(1-\alpha)}
\right\}
\right](s-t)
x'(s)\,ds.
\]
\end{definition}


\begin{remark}
The left--sided operators represent past memory effects, while the right--sided operators account for future or terminal memory interactions. Therefore, both types of hereditary behavior can be treated within the same framework.
\end{remark}


\begin{remark}
The right--sided operators may be useful in future studies on variational problems, adjoint systems, optimal control, boundary--value problems, and integration--by--parts formulas.
\end{remark}

\begin{theorem}
The proposed $\mathcal{DMFF}$ consistently recovers the classical Riemann--Liouville and Caputo fractional operators when the generator is chosen as
\(
\Phi(p,\Theta)=p.
\)
In particular:
\begin{itemize}
\item the $\mathcal{DMFI}$ reduces to the classical Riemann--Liouville fractional integral;
\item the $\mathcal{RL\;DMFD}$ reduces to the classical Riemann--Liouville fractional derivative;
\item the $\mathcal{C\;DMFD}$ reduces to the classical Caputo fractional derivative.
\end{itemize}
\end{theorem}

\begin{proof}
If $\Phi(p,\Theta)=p$, then
\[
\Psi_\alpha(t;\Theta)
=
\mathcal L^{-1}\{p^{-\alpha}\}(t)
=
\frac{t^{\alpha-1}}{\Gamma(\alpha)}.
\]
Substituting this kernel into Definitions \ref{25}--\ref{27} immediately yields the classical Riemann--Liouville and Caputo operators.
\end{proof}


\begin{theorem}
Suppose that the 
\(
\Psi_{1-\alpha}(\cdot;\Theta)
\)
form an approximate identity as
\(
\alpha\to1^-,
\)
namely,

\[
\int_0^\infty
\Psi_{1-\alpha}(s;\Theta)
f(t-s)
\,ds
\longrightarrow
f(t)
\]
for every bounded continuous function
\(
f.
\)
Then, for every
\(
x\in \mathcal{C}^1([0,\mathcal{T})],
\)

\[
\lim_{\alpha\to1^-}
{}^c \mathcal D_\Phi^\alpha x(t)
=
x'(t).
\]

\end{theorem}

\begin{proof}
From the definition of the $\mathcal{C\;DMFD}$,
\[
{}^c \mathcal D_\Phi^\alpha x(t)
=
\int_0^t
\Psi_{1-\alpha}(t-s;\Theta)
x'(s)
\,ds.
\]
Since
\(
x'
\)
is continuous on
\(
[0,\mathcal{T}],
\)
the approximate identity property implies

\[
\lim_{\alpha\to1^-}
\int_0^t
\Psi_{1-\alpha}(t-s;\Theta)
x'(s)
\,ds
=
x'(t).
\]
Therefore,
\(
\lim_{\alpha\to1^-}
{}^c \mathcal D_\Phi^\alpha x(t)
=
x'(t).
\)
\end{proof}

\begin{theorem}
Suppose that the 
\(
\Psi_\alpha(\cdot;\Theta)
\)
form an approximate identity as
\(
\alpha\to0^+,
\)
namely,

\[
\int_0^\infty
\Psi_\alpha(s;\Theta)
f(t-s)
\,ds
\to
f(t)
\]
for every bounded continuous function
\(
f.
\)
Then, for every
\(
x\in \mathcal C([0,\mathcal{T}]),
\)

\[
\lim_{\alpha\to0^+}
\mathcal I_\Phi^\alpha x(t)
=
x(t).
\]

\end{theorem}

\begin{proof}
From the definition of the $\mathcal{DMFI}$
and since
\(
x
\)
is continuous on
\(
[0,\mathcal{T}],
\)
the approximate identity property implies
\[
\lim_{\alpha\to0^+}
\int_0^t
\Psi_\alpha(t-s;\Theta)
x(s)
\,ds
=
x(t).
\]
Hence,
\(
\lim_{\alpha\to0^+}
\mathcal I_\Phi^\alpha x(t)
=
x(t).
\)

\end{proof}

\begin{remark}
The previous assumptions require the $\mathcal{DMK}$ to behave as approximate identities as
\(
\alpha\to1^-
\)
or
\(
\alpha\to0^+.
\)
Such behavior occurs for classical choices of the $\mathcal{DMG}$, including
\(
\Phi(p,\Theta)=p,
\)
hence recovering the standard limiting relations of fractional calculus.
\end{remark}


\begin{theorem}
Assume that the 
\(
\Psi_{1-\alpha}(t;\Theta)
\)
is nonnegative for all \(t>0\). Let
\(
x\in \mathcal{C}^{1}([0,\mathcal{T}])
\)
be a nondecreasing function. Then
\[
{}^{c} \mathcal D_{\Phi}^{\alpha}x(t)\ge0,
\;\; t\in[0,\mathcal{T}].
\]
Similarly, if \(x\) is nonincreasing, then
\(
{}^{c} \mathcal D_{\Phi}^{\alpha}x(t)\le 0.
\)
\end{theorem}

\begin{proof}
Suppose that \(x\) is nondecreasing. Then
\(
x'(t)\ge0,
\;\; t\in[0,\mathcal{T}].
\)
Since the generated kernel is nonnegative, we obtain
\[
\Psi_{1-\alpha}(t-s;\Theta)x'(s)\ge0,
\qquad 0\le s\le t.
\]
Therefore,
\[
{}^{c} \mathcal D_{\Phi}^{\alpha}x(t)
=
\int_{0}^{t}
\Psi_{1-\alpha}(t-s;\Theta)x'(s)\,ds
\ge0.
\]
The proof for nonincreasing functions follows analogously.
\end{proof}


\begin{remark}
The previous theorem shows that the proposed $\mathcal{DMFD}$ preserves
the qualitative monotonicity behavior of the underlying function whenever the generated memory
kernel is nonnegative. This property is important in the analysis of stability, positivity,
comparison principles, and memory--dependent dynamical systems.
\end{remark}


\subsection{Comparative Analysis and Modeling Advantages}

This subsection compares the proposed $\mathcal{DMFO}$ with several classes of generalized fractional operators, including the Caputo--Fabrizio operator, the Atangana--Baleanu operator, and fractional operators associated with Sonine kernels. These operators correspond to different memory structures, such as nonsingular exponential kernels, Mittag--Leffler kernels, and general convolution formulations.

As shown in Table~3, the proposed $\mathcal{DMFO}$ differs from these approaches because the memory kernels are generated through the $\mathcal{DMG}$
instead of being prescribed in advance. Different choices of the generator may therefore lead to different memory behaviors within the same framework.
The generated operators also retain several operational properties developed earlier, including semigroup relations, inverse formulas, and Laplace representations. This allows different classes of memory effects to be studied using the same analytical setting.

\begin{table}[h!]
\centering
\scriptsize
\renewcommand{\arraystretch}{1.25}
\begin{tabular}{|p{2.7cm}|p{3.2cm}|p{3.6cm}|p{3.6cm}|}
\hline
\textbf{Operator} &
\textbf{Main Structure} &
\textbf{Advantages} &
\textbf{Limitations}
\\
\hline

Caputo--Fabrizio &
Nonsingular exponential kernel &
Simple kernel; useful for short--memory and exponential relaxation effects; avoids singularity at the origin. &
Limited memory flexibility; exponential kernel is fixed; does not naturally recover broad classes of memory laws; semigroup property is not generally available.
\\
\hline

Atangana--Baleanu &
Nonsingular Mittag--Leffler kernel &
Captures crossover memory effects; more flexible than pure exponential kernels; suitable for several physical models. &
Kernel is still prescribed in advance; operational calculus may be more involved; semigroup and inverse relations are not automatic in the general setting.
\\
\hline

Sonine--kernel operators &
General convolution kernels satisfying Sonine--type relations &
Very general and analytically powerful; supports inverse relations and broad convolution structures. &
The kernel pair must usually be assumed or constructed separately; the memory law is not always generated from a symbolic mechanism; practical kernel selection may be difficult.
\\
\hline

 $\mathcal{DMFO}$  &
Generator--based kernel
\begin{align*}
\Psi_\alpha(t;\Theta)
=\\
\mathcal L^{-1}
\{\Phi(p,\Theta)^{-\alpha}\}(t)
\end{align*}
&
Memory is generated systematically through \(\Phi\); unifies singular, nonsingular, tempered, Mittag--Leffler, Sonine--type, and distributed--order structures; preserves semigroup and operational properties under admissibility conditions. &
Requires admissibility conditions on the generator; suitable choices of \(\Phi\) must be selected according to the modeled memory behavior.
\\
\hline
\end{tabular}

\vspace{1mm}

{\footnotesize
Table 3: Comparison between the proposed dynamic--memory fractional framework and \\some common generalized fractional operators.
}
\label{TabComparison}
\end{table}
Beyond the comparisons presented above, the proposed $\mathcal{DMFF}$ allows different memory behaviors to be considered within the same framework. In practical models, memory effects may combine short--time exponential decay, crossover behavior, long--time power--law effects, and multiscale or adaptive interactions rather than being determined by a single fixed kernel.

Such behaviors arise naturally in viscoelastic materials, anomalous diffusion processes, biological memory systems, heat transfer in complex media, and relaxation phenomena in heterogeneous environments.
However, most existing fractional operators are constructed using a prescribed fixed kernel. For instance:
\begin{itemize}[itemsep=2pt,topsep=2pt]

\item
the classical Caputo operator is associated with a singular power--law kernel;

\item
the Caputo--Fabrizio operator employs a nonsingular exponential kernel;

\item
the Atangana--Baleanu operator uses a Mittag--Leffler-type memory kernel.

\end{itemize}
Although these operators are useful in many applications, each of them is typically associated with a single dominant memory structure. Consequently, combining several hereditary mechanisms within a unified operator often becomes difficult.

\medskip

In contrast, the proposed $\mathcal{DMFF}$ generates the memory structure through the $\mathcal{DMG}$
which allows different memory behaviors to coexist simultaneously within the same operator.
To illustrate this advantage, consider the dynamic--memory fractional relaxation equation
\begin{equation}\label{Relax}
{}^{c}\mathcal D_{\Phi}^{\alpha}x(t)
=
-\kappa x(t),
\;\;
x(0)=x_{0},
\end{equation}
where \(\kappa>0\) is a relaxation parameter.
Suppose now that the $\mathcal{DMG}$ is chosen in the form
\begin{equation}\label{GeneratorExample}
\Phi(p,\Theta)
=
(p+\lambda)^{\mu}
+
\eta p^{\nu},
\end{equation}
where
\(
\lambda>0,
\;
\eta>0,
\;
0<\mu,\nu<1.
\)
The generator \eqref{GeneratorExample} combines two different hereditary mechanisms:
\begin{itemize}[itemsep=2pt,topsep=2pt]

\item
the term
\(
(p+\lambda)^{\mu}
\)
produces exponentially tempered memory effects associated with short-memory relaxation;

\item
the term
\(
p^{\nu}
\)
generates long--memory power--law hereditary behavior.

\end{itemize}
Consequently, the $\mathcal{DMK}$
\[
\Psi_{\alpha}(t;\Theta)
=
\mathcal L^{-1}
\left\{
\left(
(p+\lambda)^{\mu}
+
\eta p^{\nu}
\right)^{-\alpha}
\right\}(t)
\]
simultaneously incorporates
tempered memory attenuation,
power--law persistence,
crossover memory behavior and
multiscale hereditary interactions.
This provides an important modeling advantage over many existing fractional operators based on a single prescribed kernel.
In particular:
\begin{itemize}[itemsep=2pt,topsep=2pt]

\item
for small times, the exponentially tempered component dominates, producing rapid initial relaxation;

\item
for large times, the power--law component becomes dominant, generating persistent long-memory effects.

\end{itemize}
Therefore, the $\mathcal{DMFF}$ naturally captures the transition between short--memory and long--memory dynamics within a single symbolic operator.
By contrast, classical fractional operators usually require changing the kernel itself or introducing hybrid coupled models in order to describe such crossover phenomena.

Conversely, many existing generalized fractional operators rely on prescribed memory kernels, thereby restricting the resulting hereditary behavior to a predetermined form. For example, the classical Caputo operator captures power--law memory, the Caputo--Fabrizio operator reflects exponentially decaying memory, while the Atangana--Baleanu operator is governed by a Mittag--Leffler--type kernel.
Therefore, these operators cannot naturally generate the combined symbolic memory structure
\[
\Phi(p,\Theta)
=
(p+\lambda)^{\mu}
+
\eta p^{\nu},
\]
since this generator simultaneously incorporates exponentially tempered and power--law hereditary mechanisms within a single operational model. In particular, the corresponding $\mathcal{DMK}$
\[
\Psi_{\alpha}(t;\Theta)
=
\mathcal L^{-1}
\left\{
\left(
(p+\lambda)^{\mu}
+
\eta p^{\nu}
\right)^{-\alpha}
\right\}(t)
\]
cannot generally be represented through a single fixed exponential kernel, a single Mittag--Leffler kernel, or a pure power--law kernel.
Consequently, standard generalized fractional operators would typically require:
\begin{itemize}[itemsep=2pt,topsep=2pt]

\item
switching between different fractional operators in different temporal regimes;

\item
constructing coupled hybrid fractional models;

\item
introducing additional empirical correction terms

\end{itemize}
in order to reproduce the same crossover memory behavior generated naturally by the proposed  framework.

By contrast, the present dynamic--memory formulation incorporates these heterogeneous hereditary effects directly through the structure of the generator itself. This provides a more unified, flexible, and analytically consistent mechanism for modeling multiscale memory dynamics.

Moreover, different choices of the $\mathcal{DMG}$
 allow the dynamic memory structure to be tuned continuously according to the physical characteristics of the modeled system. Consequently, the proposed framework provides significantly greater flexibility for modeling heterogeneous and adaptive memory processes.

Figure 1 shows the qualitative behavior of several memory kernels corresponding to different choices of the $\mathcal{DMG}$. 
The figure illustrates how different generators produce transitions between exponentially decaying and power--law memory behaviors.

\begin{figure}[H]
\centering
\includegraphics[width=0.96\textwidth]{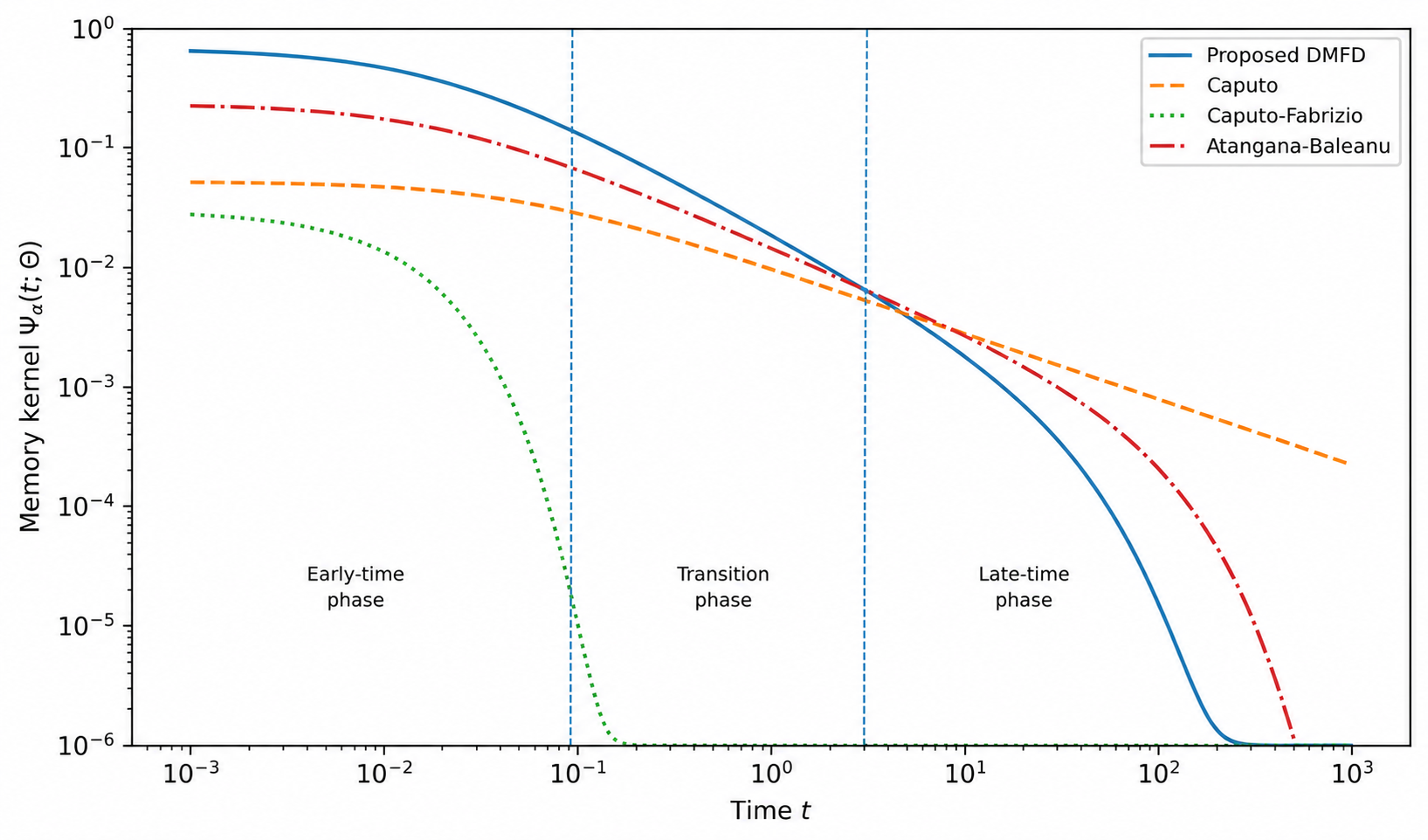}

\vspace{1pt}
\begin{minipage}{0.78\textwidth}
\captionsetup{font=scriptsize}
\caption{Qualitative comparison of generated memory kernels associated with different  choices of the $\mathcal{DMG}$. The proposed framework allows the simultaneous incorporation of exponentially tempered and power--law hereditary behaviors within a unified symbolic structure.}
\label{FigMemory}
\end{minipage}

\end{figure}
\begin{remark}
The previous example shows that suitable choices of the $\mathcal{DMG}$
lead to different combinations of memory effects within the same framework.
As a result, multiscale and heterogeneous memory processes can be modeled
while preserving the underlying operational structure.
\end{remark}


\subsection{Analytical, Differential, and Nonlocal Properties}

This subsection establishes several fundamental analytical and
calculus--type properties of the proposed dynamic--memory fractional
operators, including inverse relations, boundedness, linearity,
integration formulas, and nonlocal effects. These results illustrate
how the developed framework preserves important structures of classical
analysis while introducing new hereditary and nonlocal behaviors.

\begin{theorem}[Fundamental properties of the $\mathcal{DMFI}$ and $\mathcal{DMFD}$]
Suppose that the semigroup property holds and that the involved operators are well defined.
Then the following properties hold:
\begin{enumerate}[label=(\roman*),itemsep=2pt,topsep=2pt]
\item
\(
\mathcal D_\Phi^\alpha \mathcal I_\Phi^\alpha x=x.
\)

\item
For sufficiently smooth functions,
\(
\mathcal I_\Phi^\alpha\,{}^c \mathcal D_\Phi^\alpha x=x-x(0).
\)

\item
The operator \(\mathcal I_\Phi^\alpha\) is linear.

\item
If \(\Psi_\alpha(\cdot;\Theta)\in L^1(0,\mathcal{T})\), then
\(
\mathcal I_\Phi^\alpha:L^\infty(0,\mathcal{T})\to L^\infty(0,\mathcal{T})
\)
is bounded.

\item
Assume that the kernels \(\Psi_{1-\alpha}(\cdot;\Theta)\) form an approximate identity as
\(\alpha\to1^{-}\), that is,
\[
\int_0^t \Psi_{1-\alpha}(t-s;\Theta)f(s)\,ds \longrightarrow f(t)
\]
for every continuous function \(f\) on \([0,\mathcal{T}]\). Then, for every \(x\in \mathcal{C}^1([0,\mathcal{T}])\),
\[
\lim_{\alpha\to1^-}{}^c \mathcal{D}_\Phi^\alpha x(t)=x'(t).
\]
\end{enumerate}
\end{theorem}

\begin{proof}
Property (i) follows from the inverse relation between the $\mathcal{DMFI}$ and the corresponding $\mathcal{RL\;DMFD}$, together
with the semigroup property.

For (ii), by the definition of the $\mathcal{C\;DMFD}$,
\[
{}^c \mathcal D_\Phi^\alpha x=\mathcal I_\Phi^{1-\alpha}x'.
\]
Applying \( \mathcal I_\Phi^\alpha\) and using the semigroup property gives
\[
\mathcal I_\Phi^\alpha\,{}^c \mathcal D_\Phi^\alpha x
=
\mathcal I_\Phi^\alpha \mathcal I_\Phi^{1-\alpha}x'
=
\mathcal I_\Phi^1x'.
\]
Hence,
\[
\mathcal I_\Phi^\alpha\,{}^cD_\Phi^\alpha x
=
\int_0^t x'(s)\,ds
=
x(t)-x(0).
\]
Property (iii) follows directly from the linearity of the convolution integral defining
\(\mathcal I_\Phi^\alpha\).

To prove (iv), let \(x\in L^\infty(0,\mathcal T)\). Then
\[
|(\mathcal I_\Phi^\alpha x)(t)|
\le
\int_0^t |\Psi_\alpha(t-s;\Theta)|\,|x(s)|\,ds
\le
\|x\|_\infty \int_0^t |\Psi_\alpha(t-s;\Theta)|\,ds.
\]
Therefore,
\[
\|\mathcal I_\Phi^\alpha x\|_\infty
\le
\|\Psi_\alpha\|_{L^1(0,\mathcal T)}\|x\|_\infty,
\]
which proves the boundedness.

Finally, for (v), from the definition of the $\mathcal{C\;DMFD}$ we have
\[
{}^c \mathcal D_\Phi^\alpha x(t)
=
\int_0^t \Psi_{1-\alpha}(t-s;\Theta)x'(s)\,ds.
\]
Since \(x\in \mathcal C^1([0,\mathcal T])\), the function \(x'\) is continuous on \([0,\mathcal T]\). Hence, by the
approximate identity assumption applied to \(f=x'\), we obtain
\[
\lim_{\alpha\to1^-}{}^c \mathcal D_\Phi^\alpha x(t)
=
x'(t).
\]
This completes the proof.
\end{proof}

\begin{theorem}[Linearity of the $\mathcal{C\;DMFD}$]
Let $0<\alpha<1$ and let $x,y$ be sufficiently smooth functions. For any constants $a,b\in\mathbb R$, we have
\[
{}^c\mathcal D_\Phi^\alpha(ax+by)(t)
=
a\,{}^c\mathcal D_\Phi^\alpha x(t)
+
b\,{}^c\mathcal D_\Phi^\alpha y(t).
\]
\label{326}
\end{theorem}

\begin{proof}
Using the definition of the  $\mathcal{C\;DMFD}$  we obtain
\[
{}^c\mathcal D_\Phi^\alpha(ax+by)(t)
=
\int_0^t
\Psi_{1-\alpha}(t-s;\Theta)
(ax'(s)+by'(s))\,ds.
\]
By the linearity of the integral,
\[
{}^c\mathcal D_\Phi^\alpha(ax+by)(t)
=
a\int_0^t
\Psi_{1-\alpha}(t-s;\Theta)x'(s)\,ds
+
b\int_0^t
\Psi_{1-\alpha}(t-s;\Theta)y'(s)\,ds.
\]
Hence,
\(
{}^c\mathcal D_\Phi^\alpha(ax+by)(t)
=
a\,{}^c\mathcal D_\Phi^\alpha x(t)
+
b\,{}^c\mathcal D_\Phi^\alpha y(t).
\)
\end{proof}


\begin{theorem}[Derivative of Constants]\label{327}
Let $C$ be a constant. Then
\[
{}^c\mathcal D_\Phi^\alpha\; C=0.
\]
\end{theorem}

\begin{proof}
Since $C'=0$, we get
\[
{}^c\mathcal D_\Phi^\alpha \; C
=
\int_0^t
\Psi_{1-\alpha}(t-s;\Theta)\;C'\,ds
=
0.
\]
\end{proof}


\begin{theorem}[Product Rule]
Let $x,y\in \mathcal{C}^1([0,\mathcal{T}])$. Then
\[
{}^c\mathcal D_\Phi^\alpha(xy)(t)
=
\int_0^t
\Psi_{1-\alpha}(t-s;\Theta)
\big[
x'(s)y(s)+x(s)y'(s)
\big]\,ds.
\]
Equivalently,
\[
{}^c\mathcal D_\Phi^\alpha(xy)(t)
=
\mathcal I_\Phi^{1-\alpha}(x'y)(t)
+
\mathcal I_\Phi^{1-\alpha}(xy')(t).
\]
\end{theorem}


\begin{remark}
In general, the classical product rule does not hold for the proposed
$\mathcal{DMFD}$, namely,
\[
{}^c \mathcal D_\Phi^\alpha(xy)(t)
\neq
x(t)\,{}^c \mathcal D_\Phi^\alpha y(t)
+
y(t)\,{}^c \mathcal D_\Phi^\alpha x(t).
\]
This difference arises from the nonlocal nature of the operator, which
depends on the accumulated history through the generated kernel
\(
\Psi_{1-\alpha}.
\)
Consequently, product dynamics involve coupled hereditary effects and
cannot generally be separated as in classical local calculus.
\end{remark}


\begin{theorem}[Chain Rule]
Let $x\in \mathcal{C}^1([0,\mathcal{T}])$ and let $F\in \mathcal{C}^1(\mathbb R)$. Then
\[
{}^c\mathcal D_\Phi^\alpha(F\circ x)(t)
=
\int_0^t
\Psi_{1-\alpha}(t-s;\Theta)
F'(x(s))x'(s)\,ds.
\]
Equivalently,
\[
{}^c\mathcal D_\Phi^\alpha(F(x))(t)
=
\mathcal I_\Phi^{1-\alpha}
\big(F'(x)x'\big)(t).
\]
\end{theorem}

\begin{proof}
Since
\[
\frac{d}{ds}F(x(s))
=
F'(x(s))x'(s),
\]
the result follows directly from the definition of the $\mathcal{C\;DMFD}$ .
\end{proof}


\begin{remark}
Unlike the classical conformable derivative, the above chain rule remains nonlocal because the memory kernel $\Psi_{1-\alpha}$ acts on the whole history of $F'(x)x'$ over the interval $[0,t]$.
\end{remark}

\begin{theorem}[Integration by Parts Formula]
Let \(0<\alpha<1\), and let \(x,y\in \mathcal{C}^{1}([0,\mathcal{T}])\). Then
\[
\int_{0}^{\mathcal{T}} y(t)\,{}^{c}\mathcal D_{\Phi}^{\alpha}x(t)\,dt
=
\int_{0}^{\mathcal{T}} x'(s)
\left(
\int_{s}^{\mathcal{T}} y(t)\Psi_{1-\alpha}(t-s;\Theta)\,dt
\right)ds.
\]
\end{theorem}

\begin{proof}
Using the definition of the $\mathcal{C\;DMFD}$, we have
\[
\int_{0}^{\mathcal{T}} y(t)\,{}^{c}\mathcal D_{\Phi}^{\alpha}x(t)\,dt
=
\int_{0}^{\mathcal{T}} y(t)
\left(
\int_{0}^{t}\Psi_{1-\alpha}(t-s;\Theta)x'(s)\,ds
\right)dt.
\]
By changing the order of integration over the triangular region
\(
0\leq s\leq t\leq \mathcal{T},
\)
we obtain
\[
\int_{0}^{\mathcal{T}} y(t)\,{}^{c}\mathcal D_{\Phi}^{\alpha}x(t)\,dt
=
\int_{0}^{\mathcal{T}} x'(s)
\left(
\int_{s}^{\mathcal{T}} y(t)\Psi_{1-\alpha}(t-s;\Theta)\,dt
\right)ds.
\]
\end{proof}

\begin{remark}
The inner integral in the above formula represents a right--sided $\mathcal{DMFI}$ acting on \(y\). Thus, the integration by parts formula connects the left--sided $\mathcal{C\;DMFD}$ of \(x\) with a right--sided memory action on \(y\).
\end{remark}

\begin{theorem}[Dynamic--Memory Taylor--Type Formula]
Let \(0<\alpha<1\), and let \(x\) be sufficiently smooth on \([0,\mathcal{T}]\). Then
\[
x(t)
=
x(0)
+
\mathcal I_{\Phi}^{\alpha}
{}^{c}\mathcal D_{\Phi}^{\alpha}x(t),
\qquad t\in[0,\mathcal{T}].
\]
Equivalently,
\[
x(t)
=
x(0)
+
\int_{0}^{t}
\Psi_{\alpha}(t-s;\Theta)
{}^{c}\mathcal D_{\Phi}^{\alpha}x(s)\,ds.
\]
\end{theorem}

\begin{proof}
The result follows directly from the Caputo--type inverse formula
\[
\mathcal I_{\Phi}^{\alpha}
{}^{c}\mathcal D_{\Phi}^{\alpha}x
=
x-x(0).
\]
Therefore,
\[
x(t)
=
x(0)
+
\mathcal I_{\Phi}^{\alpha}
{}^{C}\mathcal D_{\Phi}^{\alpha}x(t),
\]
which gives the stated representation.
\end{proof}

\begin{remark}
The above representation can be regarded as a first--order dynamic--memory fractional Taylor formula. In contrast to the classical local expansion, the correction term incorporates accumulated hereditary effects through the  $\mathcal{DMK}$, thus preserving the nonlocal nature of the underlying dynamics.
\end{remark}

\begin{remark}[Dynamic--memory interpretation of Rolle--type behavior]
Let
\(
x\in \mathcal{C}^{1}([a,b])
\)
satisfy
\(
x(a)=x(b).
\)
Then, by the classical Rolle's theorem, there exists
\(
\xi\in(a,b)
\)
such that
\(
x'(\xi)=0.
\)
However, the $\mathcal{C\;DMFD}$
depends on the accumulated hereditary contribution of $x'$ over the interval $[0,t]$ through the generated kernel $\Psi_{1-\alpha}$. Consequently, stationary behavior at an isolated point does not necessarily determine the value of the $\mathcal{DMFD}$. Instead, the operator reflects a history--dependent balance of past dynamics, providing a generalized interpretation of Rolle--type behavior within the proposed $\mathcal{DMFF}$.
\end{remark}

\begin{remark}[Dynamic--memory mean value formula]
Let \(x\in C^1([a,b])\), and assume that
\[
\Psi_{1-\alpha}(t-s;\Theta)\ge 0,
\qquad a\le s\le t\le b.
\]
Then, for each fixed \(t\in(a,b]\), there exists
\(\xi_t\in(a,t)\) such that
\[
{}^c \mathcal D_{\Phi,a+}^{\alpha}x(t)
=
x'(\xi_t)
\int_a^t \Psi_{1-\alpha}(t-s;\Theta)\,ds.
\]
Equivalently,
\[
{}^c \mathcal D_{\Phi,a+}^{\alpha}x(t)
=
x'(\xi_t)\,
\mathcal I_{\Phi,a+}^{1-\alpha}1(t)
\]
provided
\(
\Psi_{1-\alpha}
\)
is continuous and nonnegative.

Thus, the $\mathcal{C\;DMFD}$ can be interpreted as
a memory--weighted mean value of the classical derivative over \([a,t]\).
Unlike the classical mean value theorem, the weighting is determined by the
generated kernel \(\Psi_{1-\alpha}\), and therefore depends on the selected
dynamic memory generator \(\Phi\).

In particular, if we choose
\(
\Phi(p,\Theta)=p,
\)
then
\[
\Psi_{1-\alpha}(t-s;\Theta)
=
\frac{(t-s)^{-\alpha}}{\Gamma(1-\alpha)}.
\]
Thus, the above formula becomes
\[
{}^c \mathcal D_{a+}^{\alpha}x(t)
=
x'(\xi_t)
\frac{(t-a)^{1-\alpha}}{\Gamma(2-\alpha)}.
\]
Passing to the limit as \(\alpha\to1^{-}\), we obtain
\[
\lim_{\alpha\to1^-}{}^c \mathcal D_{a+}^{\alpha}x(t)
=
x'(\xi_t),
\]
which is consistent with the classical mean--value interpretation. In this sense,
the dynamic--memory formula recovers the classical local behavior when the
generator is reduced to the standard choice \(\Phi(p,\Theta)=p\) and
\(\alpha\to1^{-}\).
\end{remark}

\subsection{Dynamic--Memory Fractional Derivatives of Power and Polynomial Functions}

Power functions play a fundamental role in fractional calculus and numerical approximation,
since they frequently appear in consistency analysis, interpolation procedures, spectral
methods, predictor--corrector schemes, and benchmark problems. Therefore, it is important
to characterize the action of the proposed $\mathcal{DMFO}$ on such functions.

\begin{theorem}[$\mathcal{DMFD}$ of power functions]\label{328}
Let
\[
x(t)
=
(t-a)^b,
\qquad
b>0,
\]
and suppose that the 
\(
\Psi_{1-\alpha}(\cdot;\Theta)
\)
belongs to
\(
L^1_{\mathrm{loc}}(0,\infty).
\)
Then, the $\mathcal{C\;DMFD}$ satisfies
\[
{}^c \mathcal D_\Phi^\alpha
(t-a)^b
=
b
\int_a^t
\Psi_{1-\alpha}(t-s;\Theta)
(s-a)^{b-1}
\,ds.
\]
Equivalently,

\[
{}^c \mathcal D_\Phi^\alpha
(t-a)^b
=
b
\Big(
\Psi_{1-\alpha}
*
(\cdot-a)^{b-1}
\Big)(t).
\]

\end{theorem}

\begin{proof}
For
\(
x(t)
=
(t-a)^b 
\) and by
Definition \ref{27}, we have

\[
{}^c \mathcal D_\Phi^\alpha
(t-a)^b
=
b
\int_a^t
\Psi_{1-\alpha}(t-s;\Theta)
(s-a)^{b-1}
\,ds.
\]
The convolution form follows directly from the definition of convolution.
\end{proof}

\begin{remark}
The previous result recovers several known formulas through suitable choices of the $\mathcal{DMG}$.
If
\(
\Phi(p,\Theta)
=
p,
\)
then
\[
\Psi_{1-\alpha}(t)
=
\frac{t^{-\alpha}}
{\Gamma(1-\alpha)},
\]
and consequently
\[
{}^c \mathcal D^\alpha
(t-a)^b
=
\frac{\Gamma(b+1)}
{\Gamma(b+1-\alpha)}
(t-a)^{b-\alpha},
\]
which is precisely the classical Caputo fractional derivative of power functions.

Similarly, alternative selections of
\(
\Phi
\)
generate tempered, Mittag--Leffler, nonsingular,
or other generalized dynamic--memory power laws.
Therefore, the proposed $\mathcal{DMFF}$ extends the classical power--law formulas through
the $\mathcal{DMK}$.
\end{remark}

\begin{remark}
Unlike the classical fractional derivative, where the derivative of
\(
(t-a)^b
\)
remains a pure power function, the proposed $\mathcal{DMFD}$ depends on the generated kernel
\(
\Psi_{1-\alpha}.
\)
Consequently, the resulting derivative reflects the interaction between the local growth
of the polynomial term and the hereditary structure induced by the $\mathcal{DMG}$. Different
choices of the generator therefore lead to different memory--modulated power laws.
\end{remark}


The previous result naturally extends to polynomial functions.

\begin{theorem}[$\mathcal{DMFD}$ of polynomial functions]
Let

\[
P_n(t)
=
\sum_{k=0}^{n}
a_k
(t-a)^k,
\]
where
\(
a_k\in\mathbb R,
\;
k=0,1,\ldots,n.
\)
Then
\[
{}^c \mathcal D_\Phi^\alpha
P_n(t)
=
\sum_{k=1}^{n}
k a_k
\int_a^t
\Psi_{1-\alpha}(t-s;\Theta)
(s-a)^{k-1}
\,ds.
\]
Equivalently, we get
\[
{}^c \mathcal D_\Phi^\alpha
P_n(t)
=
\sum_{k=1}^{n}
k a_k
\Big(
\Psi_{1-\alpha}
*
(\cdot-a)^{k-1}
\Big)(t).
\]

\end{theorem}

\begin{proof}
Using the linearity property established in Theorem \ref{326} together with Theorem \ref{328} applied term--by--term, we obtain
\[
{}^c \mathcal D_\Phi^\alpha
P_n(t)
=
\sum_{k=0}^{n}
a_k
{}^c \mathcal D_\Phi^\alpha
(t-a)^k.
\]
Since \(a_0\) is constant, Theorem \ref{327} implies that
\(
{}^c \mathcal D_\Phi^\alpha(a_0)=0.
\)
Hence,
\[
{}^c \mathcal D_\Phi^\alpha
P_n(t)
=
\sum_{k=1}^{n}
k a_k
\int_a^t
\Psi_{1-\alpha}(t-s;\Theta)
(s-a)^{k-1}
\,ds.
\]
The proof is complete.
\end{proof}

\begin{remark}For
\(
\Phi(p,\Theta)=p,
\)
the previous theorem reduces to

\[
{}^c \mathcal D^\alpha
P_n(t)
=
\sum_{k=1}^{n}
a_k
\frac{\Gamma(k+1)}
{\Gamma(k+1-\alpha)}
(t-a)^{k-\alpha},
\]
which coincides with the classical Caputo derivative of polynomial functions.
\end{remark}

\begin{remark}
The obtained formulas for power and polynomial functions are useful in the development
and analysis of numerical methods for the proposed framework. They may be applied in
consistency and convergence analysis, predictor--corrector schemes, spectral methods,
collocation techniques, interpolation procedures, and benchmark tests for validating
numerical approximations \cite{Diethelm2002,ZayernouriKarniadakis2013,LiZeng2015}.
\end{remark}

\begin{remark}
The previous results suggest that orthogonal polynomial families, including Legendre,
Chebyshev, Bernstein, and Jacobi polynomials, may be generalized within the present
dynamic--memory framework. Such extensions could provide efficient basis functions for
developing spectral methods adapted to heterogeneous and multiscale memory effects.
\end{remark}

\subsection{Laplace Operational Calculus}

For simplicity, throughout this subsection we write
\[
X(p)=\mathcal L\{x(t)\}(p).
\]
Since
\[
\mathcal L\{\Psi_\alpha(t;\Theta)\}(p)
=
\Phi(p,\Theta)^{-\alpha},
\]
the $\mathcal{DMFI}$ satisfies the following Laplace representation.


\begin{theorem}[Laplace Transform of the $\mathcal{DMFI}$]
The Laplace transform of the $\mathcal{DMFI}$ is
\[
\mathcal L
\{
\mathcal I_\Phi^\alpha x
\}(p)
=
\Phi(p,\Theta)^{-\alpha}X(p).
\]
\end{theorem}


\begin{theorem}[Laplace Transform of the $\mathcal{C\;DMFD}$]
Let $x$ be of exponential order. Then
\[
\mathcal L
\{
{}^c\mathcal D_\Phi^\alpha x
\}(p)
=
\Phi(p,\Theta)^{\alpha-1}
\big(
pX(p)-x(0)
\big).
\]
\end{theorem}

\begin{proof}
From the definition,
\[
{}^c\mathcal D_\Phi^\alpha x
=
\Psi_{1-\alpha}*x'.
\]
Taking the Laplace transform yields
\[
\mathcal L
\{
{}^c\mathcal D_\Phi^\alpha x
\}(p)
=
\mathcal L
\{
\Psi_{1-\alpha}
\}(p)
\mathcal L
\{
x'
\}(p).
\]
Since
\[
\mathcal L
\{
\Psi_{1-\alpha}
\}(p)
=
\Phi(p,\Theta)^{\alpha-1},
\]
and
\[
\mathcal L\{x'\}(p)
=
pX(p)-x(0),
\]
we obtain
\[
\mathcal L
\{
{}^c\mathcal D_\Phi^\alpha x
\}(p)
=
\Phi(p,\Theta)^{\alpha-1}
\big(
pX(p)-x(0)
\big).
\]
\end{proof}


\begin{remark}
If the $\mathcal{C\;DMFD}$ is defined by
\[
\mathcal L
\{
{}^c\mathcal D_\Phi^\alpha x
\}(p)
=
\Phi(p,\Theta)^\alpha X(p)
-
\Phi(p,\Theta)^{\alpha-1}x(0),
\]
then the classical Laplace variable \(p\), associated with ordinary differentiation,
is replaced by the dynamic memory generator \(\Phi\). Consequently,
different choices of \(\Phi\) produce different memory behaviors within the same
operational framework.
\end{remark}

\subsection{Dynamic--Memory Mittag--Leffler Functions}

Mittag--Leffler functions play a fundamental role in fractional calculus and memory--dependent dynamical systems, similarly to the role of the exponential function in classical differential equations. In the proposed $\mathcal{DMFF}$, the generalized Mittag--Leffler functions arise naturally through the $\mathcal{DMG}$ $\Phi$ and provide a unified representation of generalized relaxation and hereditary phenomena.

\begin{definition}
The generalized $\mathcal{DMMLF}$
associated with the $\mathcal{DMG}$ is defined by
\[
E_{\Phi,\alpha}(\lambda,t)
=
\mathcal L^{-1}
\left\{
\frac{
\Phi(p,\Theta)^{\alpha-1}
}{
\Phi(p,\Theta)^\alpha-\lambda
}
\right\}(t).
\]
\end{definition}
\noindent Mittag--Leffler functions play a central role in fractional dynamics and relaxation phenomena; see \cite{Haubold2011,Gorenflo2014}.

\begin{theorem}
The function $E_{\Phi,\alpha}(\lambda,t)$ satisfies
\[
{}^c\mathcal D_\Phi^\alpha
E_{\Phi,\alpha}(\lambda,t)
=
\lambda
E_{\Phi,\alpha}(\lambda,t).
\]
\end{theorem}

\begin{proof}
Let
\[
X(p)
=
\frac{
\Phi(p,\Theta)^{\alpha-1}
}{
\Phi(p,\Theta)^\alpha-\lambda
}.
\]
Using the Laplace--transform formula for $\mathcal{C\;DMFD}$, we obtain
\[
\mathcal L
\left\{
{}^c\mathcal D_\Phi^\alpha
E_{\Phi,\alpha}(\lambda,t)
\right\}(p)
=
\Phi(p,\Theta)^\alpha X(p)
-
\Phi(p,\Theta)^{\alpha-1}E_{\Phi,\alpha}(\lambda,0).
\]
Substituting the expression for $X(p)$ yields
\[
\mathcal L
\left\{
{}^c\mathcal D_\Phi^\alpha
E_{\Phi,\alpha}(\lambda,t)
\right\}(p)
=
\lambda
\frac{
\Phi(p,\Theta)^{\alpha-1}
}{
\Phi(p,\Theta)^\alpha-\lambda
}.
\]
Hence,
\[
\mathcal L
\left\{
{}^c\mathcal D_\Phi^\alpha
E_{\Phi,\alpha}(\lambda,t)
\right\}(p)
=
\lambda
\mathcal L
\left\{
E_{\Phi,\alpha}(\lambda,t)
\right\}(p).
\]
Applying the inverse Laplace transform completes the proof.
\end{proof}

\begin{theorem}[Reduction to the Classical Mittag--Leffler Function]
If
\(
\Phi(p,\Theta)=p,
\)
then
\[
E_{\Phi,\alpha}(\lambda,t)
=
E_\alpha(\lambda t^\alpha),
\]
where $E_\alpha$ denotes the classical one--parameter Mittag--Leffler function.
\end{theorem}

\begin{proof}
If $\Phi(p,\Theta)=p$, then
\[
E_{\Phi,\alpha}(\lambda,t)
=
\mathcal L^{-1}
\left\{
\frac{
p^{\alpha-1}
}{
p^\alpha-\lambda
}
\right\}(t),
\]
which is precisely the classical Laplace--transform representation of the Mittag--Leffler function
\(
E_\alpha(\lambda t^\alpha).
\)
\end{proof}

\begin{remark}
Different choices of the $\mathcal{DMG}$  produce different classes of generalized relaxation functions, thereby allowing the proposed framework to model a wide spectrum of hereditary and multiscale dynamical behaviors.
\end{remark}


\section{Application to a Nonlinear $\mathcal{DMFLS}$}

In this section, we illustrate the applicability of the $\mathcal{DMFF}$ through a class of nonlinear $\mathcal{DMFLS}$ involving coupled memory interactions, nonlinear forcing terms, and nonlocal effects. The considered model extends several existing fractional Langevin equations by incorporating generator--driven memory structures together with mixed dynamic--memory fractional operators.

\subsection{Model Formulation and Problem Description}

We consider the nonlinear $\mathcal{DMFLS}$ of the form
\begin{equation}\label{MainSystem}
{}^{c}\mathcal D_{\Phi}^{\alpha}
\Big(
{}^{c}\mathcal D_{\Phi}^{\beta}x(t)
+
\lambda(t)x(t)
\Big)
=
\mathcal{A} x(t)
+
\mathcal{B}\,x(\sigma(t))
+
\mathcal{F}\big(t,x(t),x(\sigma(t))\big),
\quad t\in[0,\mathcal{T}],
\end{equation}
subject to the nonlocal initial conditions
\begin{equation}\label{IC}
x(0)=x_{0},
\qquad
{}^{c}\mathcal D_{\Phi}^{\beta}x(0)=x_{1}.
\end{equation}
The unknown function
\(
x:[0,\mathcal{T}]\to\mathbb R^{d}
\)
denotes the state vector, while
\(
\alpha,\beta\in(0,1)
\)
represent the dynamic--memory fractional orders. The function
\(
\lambda:[0,\mathcal{T}]\to\mathbb R
\)
is a continuous damping coefficient, and
\(
\mathcal{A},\mathcal{B}\in\mathbb R^{d\times d}
\)
are constant matrices describing linear interactions. Moreover, the mapping
\(
\sigma:[0,\mathcal{T}]\to[0,\mathcal{T}]
\)
denotes a deviating argument that may account for delay, proportional delay,
or more general hybrid memory effects, whereas
\(
\mathcal{F}:[0,\mathcal{T}]\times\mathbb R^{d}\times\mathbb R^{d}\to\mathbb R^{d}
\)
represents a nonlinear forcing term. Fractional Langevin equations have been studied in connection with anomalous diffusion, relaxation processes, and systems with memory effects; see, for example, \cite{Lutz2001,Alzabut2020}.

The proposed model combines several features, including nested $\mathcal{DMFO}$, memory effects generated through the $\mathcal{DMG}$,
Langevin--type damping terms, nonlinear forcing functions, and deviating arguments representing delayed or hybrid interactions. Therefore, system \eqref{MainSystem} includes several existing fractional Langevin models as particular cases, such as models with classical, tempered, nonsingular, or other generalized memory structures.
\subsection{Theoretical Analysis}

The analysis of system \eqref{MainSystem} requires suitable tools to handle the nonlinear
interactions and generated memory effects arising in the proposed framework.
To establish solvability results, we reformulate the problem as an equivalent
nonlinear Volterra integral equation, which enables the application of fixed
point techniques for proving existence and uniqueness of solutions.

To study solvability, we first rewrite the proposed $\mathcal{DMFLS}$ as an equivalent nonlinear integral equation. This allows the problem to be analyzed within a Volterra integral setting and enables the use of fixed point methods for establishing existence and uniqueness results.

Applying $\mathcal I_\Phi^\alpha$ to both sides of \eqref{MainSystem} and using the inverse relations established in Section 3, we obtain
\[
{}^{c}\mathcal D_{\Phi}^{\beta}x(t)
+
\lambda(t)x(t)
=
x_{1}
+
\mathcal I_{\Phi}^{\alpha}
\Big(
\mathcal A x
+
\mathcal{B}\,x(\sigma)
+
\mathcal{F}(t,x,x(\sigma))
\Big)(t).
\]
Applying
\(
\mathcal I_{\Phi}^{\beta}
\)
again yields the equivalent nonlinear integral equation
\begin{align}
x(t)
=
x_{0}
&-
\mathcal I_{\Phi}^{\beta}
(\lambda x)(t)
+
\mathcal I_{\Phi}^{\alpha+\beta}
\Big(
\mathcal A x
+
\mathcal{B}\,x(\sigma)
+
\mathcal{F}(t,x,x(\sigma))
\Big)(t)
+
x_{1}\Psi_{\beta}(t;\Theta).
\label{IntegralForm}
\end{align}
Equation (\ref{IntegralForm}) rewrites the original $\mathcal{DMFLS}$ as a nonlinear Volterra--type integral equation, where the memory effects depend on the generated kernels. Therefore, different classes of memory laws can be analyzed through the same solvability approach by varying the $\mathcal{DMG}$.

To investigate existence and uniqueness, we formulate the problem as a fixed point equation in a suitable Banach space
\[
\mathcal{X}=\mathcal{C}([0,\mathcal{T}],\mathbb R^{d}),
\]
which is equipped with the norm
\[
\|x\|_{X}
=
\sup_{t\in[0,\mathcal{T}]}
\|x(t)\|.
\]

Define the operator
\[
( \mathbb{T} x)(t)
=
x_{0}
-
\mathcal I_{\Phi}^{\beta}
(\lambda x)(t)
+
\mathcal I_{\Phi}^{\alpha+\beta}
\Big(
\mathcal{A}x
+
\mathcal{B}\,x(\sigma)
+
\mathcal{F}(t,x,x(\sigma))
\Big)(t)
+
x_{1}\Psi_{\beta}(t;\Theta).
\]
Then solutions of system \eqref{MainSystem} correspond exactly to fixed points of the operator \(\mathbb{T}\).

\begin{theorem}[Existence under Lipschitz--Type Conditions]
Assume that:
\begin{itemize}[itemsep=2pt,topsep=2pt]
\item[]
\begin{itemize}
\item[(H1)]
\(\mathcal{F}(t,u,v)\) is continuous on
\(
[0,\mathcal{T}]\times\mathbb R^{d}\times\mathbb R^{d};
\)

\item[(H2)]
there exists \(L>0\) such that
\[
\|\mathcal{F}(t,u_{1},v_{1})-\mathcal{F}(t,u_{2},v_{2})\|
\leq
L
\big(
\|u_{1}-u_{2}\|
+
\|v_{1}-v_{2}\|
\big);
\]

\item[(H3)]
the $\mathcal{DMK}$ satisfy the semigroup property;

\item[(H4)]
the operators
\(
\mathcal I_{\Phi}^{\alpha},
\;
\mathcal I_{\Phi}^{\beta},
\;
\mathcal I_{\Phi}^{\alpha+\beta}
\)
are bounded on \(\mathcal{X}\).
\end{itemize}
\end{itemize}
Then system \eqref{MainSystem} admits at least one solution on \([0,\mathcal{T}]\).
\end{theorem}

\begin{proof}
The proof follows from Schauder's fixed point theorem applied to the operator generated by the right--hand side of \eqref{IntegralForm}, together with the continuity and compactness properties induced by the generated $\mathcal{DMK}$.
\end{proof}

\begin{theorem}[Existence of Mild Solutions]
Assume that:
\begin{itemize}[itemsep=2pt,topsep=2pt]
\item[]
\begin{itemize}

\item[(G1)]
The nonlinear function
\(
\mathcal{F}:[0,\mathcal{T}]\times\mathbb R^{d}\times\mathbb R^{d}\to\mathbb R^{d}
\)
is continuous;

\item[(G2)]
There exists a constant \(\mathcal{M}>0\) such that
\[
\|\mathcal{F}(t,u,v)\|
\le
\mathcal{M}
\big(
1+\|u\|+\|v\|
\big);
\]

\item[(G3)]
The generated kernels
\(
\Psi_{\alpha},
\Psi_{\beta},
\Psi_{\alpha+\beta}
\)
belong to
\(
L^{1}(0,\mathcal{T});
\)

\item[(G4)]
The deviating argument
\(
\sigma:[0,\mathcal{T}]\to[0,\mathcal{T}]
\)
is continuous.

\end{itemize}
\end{itemize}
Then system \eqref{MainSystem} admits at least one mild solution on \([0,\mathcal{T}]\).
\end{theorem}

\begin{proof}
Using the integral representation \eqref{IntegralForm}, define the operator
\[
\mathbb{T}:\mathcal{X} \to \mathcal{X}
\]
by
\[
(\mathbb{T} x)(t)
=
x_{0}
-
\mathcal I_{\Phi}^{\beta}
(\lambda x)(t)
+
\mathcal I_{\Phi}^{\alpha+\beta}
\Big(
\mathcal{A}x
+
\mathcal{B}\,x(\sigma)
+
\mathcal{F}(t,x,x(\sigma))
\Big)(t)
+
x_{1}\Psi_{\beta}(t;\Theta).
\]
The continuity of \(\mathcal{F}\), together with the boundedness of the generated fractional integral operators, implies that \(\mathbb{T}\) is continuous on \(\mathcal{X}\). Moreover, the compactness properties induced by the generated memory kernels and the Arzelà--Ascoli theorem imply that \(\mathbb{T}\) maps bounded subsets of \(\mathcal{X}\) into relatively compact subsets.

Therefore, Schauder's fixed point theorem guarantees the existence of at least one fixed point of \(\mathbb{T}\). Hence, system \eqref{MainSystem} admits at least one mild solution on \([0,\mathcal{T}]\).
\end{proof}

\begin{theorem}[Uniqueness of Solutions]
Assume that hypotheses \((H1)-(H4)\) hold and additionally
\[
\|
\mathcal I_{\Phi}^{\alpha+\beta}
\|
\big(
\|\mathcal{A}\|
+
\|\mathcal{B}\|
+
L
\big)
+
\|
\mathcal I_{\Phi}^{\beta}
\|
\|\lambda\|
<1.
\]
Then system \eqref{MainSystem} admits a unique solution on \([0,\mathcal{T}]\).
\end{theorem}

\begin{proof}
The proof follows from the Banach contraction principle applied to the nonlinear operator associated with \eqref{IntegralForm}.
\end{proof}

\subsection{Numerical Approximation Scheme}

In this subsection, we develop a predictor--corrector numerical scheme for approximating solutions of the $\mathcal{DMFLS}$ \eqref{MainSystem}. The proposed approach adapts the classical Adams--Bashforth--Moulton methodology to the generator--based dynamic--memory setting. While existing predictor--corrector methods for fractional differential equations \cite{Diethelm2002,Garrappa2015} are mainly designed for prescribed memory kernels, the present formulation accommodates memory effects generated through the $\mathcal{DMG}$.

The numerical approximation of dynamic--memory fractional systems is generally
more challenging than that of classical fractional equations, since the generated
kernels
\[
\Psi_\alpha(t;\Theta)
=
\mathcal L^{-1}
\{
\Phi(p,\Theta)^{-\alpha}
\}(t)
\]
may exhibit singular, nonsingular, tempered, oscillatory, or multiscale memory
behaviors. Consequently, numerical schemes designed for fixed power--law kernels
cannot always be applied directly.

To address this issue, we develop a predictor--corrector approach based on the
generated $\mathcal{DMK}$ rather than on a prescribed memory law. The proposed
scheme extends the classical Adams--Bashforth--Moulton framework to generator--based
$\mathcal{DMFO}$ , allowing different classes of memory effects to be treated
within a unified computational procedure. In contrast to many existing methods,
which often require separate formulations for different kernels, the present approach
naturally accommodates singular, nonsingular, tempered, Mittag--Leffler, hybrid, and
multiscale memory structures while preserving the convolution and semigroup properties
of the continuous model.

Let
\[
0=t_{0}<t_{1}<\cdots<t_{N}=\mathcal{T},
\]
where
\(
t_{n}=nh,
\;
h=\frac{\mathcal{T}}{N}.
\)
Using the equivalent integral representation \eqref{IntegralForm}, we write
\begin{align}
x(t_{n+1})
=
x_{0}
&-
\int_{0}^{t_{n+1}}
\Psi_{\beta}(t_{n+1}-s;\Theta)
\lambda(s)x(s)\,ds
\nonumber\\
&+
\int_{0}^{t_{n+1}}
\Psi_{\alpha+\beta}(t_{n+1}-s;\Theta)
G(s,x(s))\,ds
\nonumber\\
&+
x_{1}\Psi_{\beta}(t_{n+1};\Theta),
\label{NumericalIntegral}
\end{align}
where
\(
G(s,x(s))
=
\mathcal{A}x(s)
+
\mathcal{B}\,x(\sigma(s))
+
\mathcal{F}(s,x(s),x(\sigma(s))).
\)
The predictor approximation is given by
\begin{align}
x_{n+1}^{P}
=
x_{0}
&-
h
\sum_{j=0}^{n}
\Psi_{\beta}(t_{n+1}-t_{j};\Theta)
\lambda(t_{j})x_{j}
\nonumber\\
&+
h
\sum_{j=0}^{n}
\Psi_{\alpha+\beta}(t_{n+1}-t_{j};\Theta)
G(t_{j},x_{j})
\nonumber\\
&+
x_{1}\Psi_{\beta}(t_{n+1};\Theta).
\end{align}
The corrected approximation is computed by
\begin{align}
x_{n+1}
=
x_{0}
&-
\frac{h}{2}
\sum_{j=0}^{n+1}
a_{j,n+1}^{(\beta)}
\Psi_{\beta}(t_{n+1}-t_{j};\Theta)
\lambda(t_{j})x_{j}
\nonumber\\
&+
\frac{h}{2}
\sum_{j=0}^{n+1}
a_{j,n+1}^{(\alpha+\beta)}
\Psi_{\alpha+\beta}(t_{n+1}-t_{j};\Theta)
G(t_{j},x_{j})
\nonumber\\
&+
x_{1}\Psi_{\beta}(t_{n+1};\Theta),
\end{align}
where
\(
a_{j,n+1}^{(\gamma)}
\)
denote the Adams--Moulton correction coefficients corresponding to the generated memory kernels.

The numerical approximation is computed iteratively. At each step, the predictor value
\(
x_{n+1}^{P}
\)
is obtained from previously computed states, while the corrected value incorporates updated information through the weighted Adams--Moulton coefficients. Since the coefficients
\(
a_{j,n+1}^{(\gamma)}
\)
depend on the generated kernels, they reflect the memory behavior determined by the $\mathcal{DMG}$.

Therefore, the numerical scheme adapts to the selected memory generator and can be applied to systems involving heterogeneous, adaptive, or crossover memory effects.


\begin{remark}
Unlike many existing predictor--corrector methods designed for prescribed memory
kernels, the proposed approach is driven by the generated memory structure induced
by the $\mathcal{DMG}$. As a result, the same numerical framework can treat diverse
memory behaviors within a unified computational setting, without requiring
kernel--specific reformulations.
\end{remark}
\begin{remark}
A rigorous analysis of the convergence and stability of the proposed predictor--corrector scheme depends on the admissibility properties of the generated kernels and the regularity assumptions imposed on the $\mathcal{DMG}$. The investigation of these properties is left for future work.
\end{remark}

\subsection{Recovery and Generation of Existing Fractional Langevin Models}
The proposed $\mathcal{DMFF}$ can recover several existing fractional Langevin equations as special cases through appropriate choices of the $\mathcal{DMG}$ $\Phi$. It also allows the construction of more general memory--dependent Langevin systems, including models involving classical, tempered, nonsingular, Mittag--Leffler, and hybrid memory effects, within a common framework.

To illustrate this mechanism, we consider several representative choices of the $\mathcal{DMG}$ and examine the corresponding reductions of the proposed  $\mathcal{DMFLS}$. These examples demonstrate how different memory generators induce distinct hereditary laws and therefore produce different classes of Langevin dynamics while preserving a common analytical structure.


\begin{example}[Recovery of the classical fractional Langevin model]

Consider the $\mathcal{DMG}$

\[
\Phi(p,\Theta)=p.
\]
Then the generated $\mathcal{DMK}$ becomes

\[
\Psi_\alpha(t;\Theta)
=
\mathcal L^{-1}
\{
p^{-\alpha}
\}(t)
=
\frac{t^{\alpha-1}}
{\Gamma(\alpha)}.
\]
Hence, the $\mathcal{DMFI}$ reduces to the classical
Riemann--Liouville fractional integral

\[
(\mathcal I_\Phi^\alpha x)(t)
=
\frac1{\Gamma(\alpha)}
\int_0^t
(t-s)^{\alpha-1}
x(s)\,ds,
\]
while the associated $\mathcal{DMFD}$ recovers the standard
Riemann--Liouville and Caputo fractional derivatives.
Consequently, the proposed $\mathcal{DMFLS}$ recovers classical fractional Langevin dynamics with power--law hereditary effects; see \cite{Kou2008}.

\end{example}


\begin{example}[Recovery of exponentially tempered Langevin models]

Let

\[
\Phi(p,\Theta)
=
p+\lambda,
\;\;
\lambda>0.
\]
Then
\[
\Psi_\alpha(t;\Theta)
=
\mathcal L^{-1}
\{
(p+\lambda)^{-\alpha}
\}(t)
=
\frac{
e^{-\lambda t}
t^{\alpha-1}
}{
\Gamma(\alpha)
}.
\]
Therefore,
\[
(\mathcal I_\Phi^\alpha x)(t)
=
\frac1{\Gamma(\alpha)}
\int_0^t
e^{-\lambda(t-s)}
(t-s)^{\alpha-1}
x(s)\,ds.
\]
Thus, the resulting Langevin model exhibits exponentially attenuated hereditary effects, similar to those arising in generalized and tempered fractional Langevin dynamics; see \cite{Deng2009}.

\end{example}


\begin{example}[Recovery of generalized Mittag--Leffler memory]

Suppose that the $\mathcal{DMG}$ generates kernels satisfying

\[
\Psi_\alpha(t;\Theta)
=
t^{\alpha-1}
E_{\mu,\alpha}
(-\lambda t^\mu),
\]
where
\(
E_{\mu,\alpha}
\)
denotes the two--parameter Mittag--Leffler function.
Then the corresponding dynamic--memory operators generate
nonlocal interactions involving crossover and multiscale
memory effects.

Consequently, the resulting $\mathcal{DMFLS}$
extends fractional Langevin models governed by Mittag--Leffler hereditary structures; see \cite{Wang2018}.

\end{example}


The above examples demonstrate that different selections of the   $\mathcal{DMG}$ lead naturally to distinct classes of fractional Langevin models,
including classical power--law memory, exponentially tempered memory,
and generalized crossover hereditary effects. Therefore, the proposed
$\mathcal{DMFF}$ provides a unified mechanism for constructing and analyzing broad
families of fractional Langevin systems through the structure of the
generated memory kernel itself.



\section{Conclusion}
This work developed a generator--based dynamic--memory framework for fractional calculus in which hereditary effects are generated through a dynamic memory generator rather than prescribed by fixed kernels. Unlike many existing generalized fractional formulations, the proposed approach provides a unified mechanism for constructing broad classes of memory--dependent operators while preserving a consistent analytical structure.

Within this setting, generalized dynamic--memory fractional integrals together with Riemann--Liouville and $\mathcal{C\;DMFD}$ were introduced through inverse Laplace representations. Several analytical and operational properties were established, including admissibility conditions, semigroup structures, inverse relations, composition formulas, boundedness results, and Laplace operational representations. The framework also led naturally to generalized dynamic--memory Mittag--Leffler functions associated with the underlying generators.

Explicit formulas for the dynamic--memory fractional derivatives of power and polynomial functions were further derived. These results extend classical fractional power laws through generated memory kernels and provide useful tools for consistency analysis, approximation techniques, predictor--corrector schemes, and the development of numerical methods adapted to heterogeneous memory structures.

Another important feature of the developed theory is its ability to recover numerous classical and modern fractional operators through suitable choices of the dynamic memory generator. Consequently, singular, nonsingular, tempered, Mittag--Leffler, Sonine--type, and other hereditary behaviors can be described within the same analytical framework.

The applicability of the proposed approach was illustrated through a nonlinear dynamic--memory fractional Langevin system involving coupled memory effects, nonlinear interactions, and generalized hereditary behavior. The obtained formulation shows that different memory mechanisms, including crossover and multiscale effects, may be incorporated within a single symbolic structure without requiring separate kernel constructions.

Several directions remain open for future investigation. These include stochastic dynamic--memory models, variable--order and adaptive generators, numerical methods for generator--dependent operators, fractional optimal control, and partial differential equations involving dynamic memory effects. Further studies may also address spectral analysis, inverse problems, admissibility theory, and applications to anomalous diffusion, viscoelasticity, and biological systems.

\section*{Declarations}

\textbf{Funding}
The author received no specific funding for this work.\\
\textbf{Conflict of interest}
The author declares that he has no conflict of interest.\\
\textbf{Data availability}
No datasets were generated or analyzed during the current study.\\
\textbf{Ethics approval}
This article does not contain any studies involving human participants or animals performed by any of the authors.\\
\textbf{Use of artificial intelligence tools}
The author declares that no artificial intelligence tools were used in the development of the mathematical results, proofs, or scientific conclusions presented in this manuscript. 
\\
\textbf{Acknowledgment} The author would like to thank Prince Sultan University and Ostim Technical University for their help and support.


\end{document}